\documentclass[12pt]{amsart}
\usepackage{epsfig,color}

\usepackage{mathrsfs}

\theoremstyle{definition}
\allowdisplaybreaks

\newtheorem{theorem}{Theorem}[section]

\newtheorem{proposition}[theorem]{Proposition}
\newtheorem{lemma}[theorem]{Lemma}
\newtheorem{remark}[theorem]{Remark}
\newtheorem{corollary}[theorem]{Corollary}

\numberwithin{equation}{section}
\newcommand\nn{\nonumber}
\usepackage{enumerate}

\DeclareMathOperator*{\im}{im}

\newcommand{\pr}{\partial}

\newcommand{\Lap}{\Delta}

\newcommand{\cL}{\mathcal{L}}
\newcommand{\cD}{\mathcal{D}}

\newcommand{\R}{\mathbb{R}}

\usepackage{comment}

\title{Rigidity and \L ojasiewicz inequalities for Clifford self-shrinkers}

\author{Ao Sun}
\address{Department of Mathematics,
University of Chicago,
5734 S. University Avenue,
Chicago, IL 60637, USA}
\email{aosun@uchicago.edu}

\author{Jonathan J. Zhu}
\address{Mathematical Sciences Institute, Australian National University, Hanna Neumann Building, Science Road, Canberra, ACT 2601, Australia and Department of Mathematics, Princeton University, Fine Hall, Washington Road, Princeton, NJ 08544, USA}
\email{jjzhu@math.princeton.edu}

\usepackage{lineno}

\begin{document}
	\begin{abstract}
		We show that the product of two round shrinking spheres is an isolated self-shrinker in any codimension, modulo rotations. Moreover we prove explicit \L ojasiewicz inequalities near such products. \L ojasiewicz inequalities were previously used by Schulze to prove uniqueness of tangent mean curvature flows at compact shrinkers; our results provide an explicit rate of convergence to products of two spheres. 
	\end{abstract}
	\date{\today}
	\maketitle

	\section{Introduction}
	A submanifold $\Sigma^n$ in the Euclidean space $\R^N$ is called a self-shrinker (or shrinker for short) if it satisfies the following equation
	\begin{equation}\label{eq:phi}
		\phi_\Sigma:=\frac{x^\perp}{2}-\mathbf{H}=0.
	\end{equation}
	Here $x$ is the position vector and $\mathbf{H}$ is the mean curvature vector, defined as the negative trace of the second fundamental form. The simplest examples of self-shrinkers $\Sigma^n\subset \R^N$ are round spheres $\mathbb{S}^n_{\sqrt{2n}}$ of radius $\sqrt{2n}$. Any Cartesian product of shrinkers is also a shrinker. 
	
	The goal of this paper is to show a local rigidity for the product of two shrinking spheres in any codimension. (When considered as hypersurfaces in a larger sphere, such products have also been referred to as Clifford minimal hypersurfaces.)
	
	\begin{theorem}\label{thm:main-rigidity-intro}
		Given $\alpha, k_1,k_2$, there exists $\epsilon>0$ such that if $\Sigma' \subset \mathbb{R}^N$ is a self-shrinker that is $(C^{2,\alpha},\epsilon)$-close to $\Sigma = \mathbb{S}^{k_1}_{\sqrt{2k_1}} \times \mathbb{S}^{k_2}_{\sqrt{2k_2}}$, then $\Sigma'$ is a rotation of $\Sigma$.
	\end{theorem}
	
	The special case of the Clifford torus $k_1=k_2=1$, $N=4$, was recently proven by Evans, Lotay and Schulze \cite{ELS18}. Theorem \ref{thm:main-rigidity-intro} is a direct consequence of the following quantitative rigidity theorem:
	
	\begin{theorem}
		\label{thm:quant-rigidity-intro}
		Given $\alpha, k_1,k_2$, there exist $C,\epsilon>0$ such that if $\Sigma' \subset \mathbb{R}^N$ is any submanifold $(C^{2,\alpha},\epsilon)$-close to $\Sigma = \mathbb{S}^{k_1}_{\sqrt{2k_1}} \times \mathbb{S}^{k_2}_{\sqrt{2k_2}}$, then there is a rotation $\mathcal{R} \in \mathrm{SO}(N)$ such that $\mathcal{R}\cdot\Sigma'$ may be written as the graph of a normal vector field $V$ on $\Sigma$, where $V$ is orthogonal to rotations and satisfies
		\begin{equation}
			\|V\|_{C^{2,\alpha}}^3
			\leq C \|\phi_{\Sigma'}\|_{C^{0,\alpha}},
		\end{equation}
		\begin{equation}
			\|V\|_{L^2}^3
			\leq C \|\phi_{\Sigma'}\|_{L^2}.
		\end{equation}
	\end{theorem}
	
	Here `rotations' refers to the normal vector fields induced by ambient rotation. Theorem \ref{thm:quant-rigidity-intro} represents a \L ojasiewicz inequality of the first kind (`distance \L ojasiewicz'; c.f. \cite{CM15} and \cite{CM19}) with explicit exponent for the Gaussian area functional 
	
	\begin{equation}\label{eq:F is Gaussian}
		\mathcal{F}(\Sigma)= (4\pi)^{-\frac{n}{2}}\int_{\Sigma}e^{-\frac{|x|^2}{4}}.
	\end{equation}
	The shrinker quantity $\phi$ is the $L^2$-gradient of $\mathcal{F}$, so self-shrinkers are critical points of $\mathcal{F}$ and Theorem \ref{thm:quant-rigidity-intro} bounds the distance to $\Sigma$ by the explicit power $\frac{1}{3}$ of the gradient. Note that Schulze \cite{S14} showed that the \L ojasiewicz-Simon inequality proven by L. Simon \cite{S83} is applicable for closed shrinkers, which yields the following estimate (see also \cite{CM18}):
	
	\begin{theorem}[\cite{S14}]
		\label{thm:schulze}
		Let $\Sigma$ be a closed shrinker, then there exists $\epsilon>0$ and $\beta\in (\frac{1}{2},1)$ so that if $\Sigma'$ is any submanifold $(C^{2,\alpha},\epsilon)$-close to $\Sigma$, then \begin{equation}|\mathcal{F}(\Sigma')- \mathcal{F}(\Sigma)|^\beta \leq \|\phi_{\Sigma'} \|_{L^2 }. \end{equation}
	\end{theorem}
	This is a \L ojasiewicz inequality of the second kind (`gradient \L ojasiewicz'). As explained in \cite[Section 0.4]{CM15}, one expects the first kind to imply the second. Indeed, from Theorem \ref{thm:quant-rigidity-intro} we are able to prove a gradient \L ojasiewicz inequality with explicit exponent:
	
	\begin{theorem}
		\label{thm:gradient-L}
		Let $\Sigma = \mathbb{S}^{k_1}_{\sqrt{2k_1}} \times \mathbb{S}^{k_2}_{\sqrt{2k_2}}$. Then Theorem \ref{thm:schulze} holds with $\beta= \frac{3}{4}$. 
	\end{theorem}
	
	Schulze \cite{S14} showed that such an inequality implies uniqueness of tangent flows, and from his work we immediately deduce an explicit convergence rate of the mean curvature flow near a singularity modelled by $\Sigma=\mathbb{S}^{k_1}_{\sqrt{2k_1}} \times \mathbb{S}^{k_2}_{\sqrt{2k_2}}$. Specifically, if the rescaled mean curvature flow $\widetilde{M_\tau}$ converges to $\Sigma$ as $\tau\to \infty$, then when $\tau$ is sufficiently large, $\widetilde{M_\tau}$ can be written as the graph of a vector field $V(\tau)$ over $\Sigma$, with 
	\begin{equation}\|V(\tau)\|_{L^2}\leq C\tau^{-1/2}.\end{equation} 

	\subsection{Background}
	Self-shrinkers $\Sigma^n \subset \mathbb{R}^N$ can also be viewed as minimal submanifolds in the Gaussian space $(\R^N,e^{-\frac{|x|^2}{4n}}\delta)$. Thus, we can also interpret our main theorem as follows: products of two spheres are isolated critical points of the Gaussian functional $\mathcal F$ in the space of submanifolds. 
	
	The local rigidity of self-shrinkers is strongly related to the singular behaviour of mean curvature flows, and has thus been a subject of significant interest. In \cite{CIM15}, Colding-Ilmanen-Minicozzi proved that round cylinders in codimension 1 are locally rigid (up to rotations). Later in \cite{GZ18}, Guang-Zhu removed a technical assumption. 
	
	Higher codimension creates new technical difficulties, and also admits self-shrinkers which are geometrically more complicated; see the discussions in \cite{ALW14}, \cite{AS13}, \cite{LL15}. In \cite{ELS18}, Evans-Lotay-Schulze proved the local rigidity of Clifford torus as a self-shrinker in $\R^4$. More recently, in \cite{CM19} Colding-Minicozzi proved that round cylinders are also rigid in any codimension. 
	
	\L ojasiewicz inequalities have become a popular method for studying uniqueness problems related to geometric flows. There have been many applications following the method of Simon \cite{S83}, which uses Lyapunov-Schmidt reduction to classical inequalities of \L ojasiewicz \cite{Loj}. For mean curvature flow these include uniqueness of blowups at compact shrinkers \cite{S14} and at asymptotically conical shrinkers \cite{CS19}. For round cylinders, however, Colding-Minicozzi \cite{CM15, CM19} prove \L ojasiewicz inequalities directly, which also imply rigidity. In this paper, we also take the direct approach using a perturbative analysis. The second author used the same strategy to show uniqueness and rigidity at a more general class of cylinders \cite{Z20}. 
	
	We remark that our analysis is inherently local, that is, we do not give an explicit neighbourhood in which the Clifford shrinkers are isolated. On the other hand, there are many results on the rigidity of self-shrinkers under explicit curvature or other conditions, see \cite{S05}, \cite{CL13}, \cite{DX14}, \cite{D18}. We also mention that there are various contexts in which the rigidity of Clifford hypersurfaces is also an attractive topic: \cite{B13} for embedded minimal tori in $S^3$ and \cite{L02} for Willmore surfaces in a sphere. Most famously the rigidity of Clifford hypersurfaces as minimal hypersurfaces in a sphere is known as Chern's conjecture; see \cite{L69, S68, CCK} or \cite{XX17} and references therein, which includes connections to the shrinker setting. 
	
	\subsection{Expansion strategy}
	Our proof strategy is a variant of the method by Evans-Lotay-Schulze \cite{ELS18}, which was also used by the second author in \cite{Z20}. We do so directly in the setting of an elliptic integrand $\mathcal{F}$ defined on a Banach space $\mathcal{E}$. Suppose $0\in \mathcal{E}$ is a critical point of $F$, and $\mathcal{G}:\mathcal{E}\to\mathcal{E}'$ is the Euler-Lagrange operator of $\mathcal{F}$. In the shrinker setting, $\mathcal{E}$ will be the space of normal variations orthogonal to those generated by rotations, $\mathcal{F}$ the Gaussian area as above and $\mathcal{G}$ the shrinker mean curvature $\phi$. 
	
		The strategy, as explained in \cite[Section 1.3]{Z20}, is to attempt to invert the series expansion of $\mathcal{G}$. Here we give a more `hands on' explanation, with the goal of highlighting some subtleties in our infinite-dimensional geometric setting and some differences with \cite{ELS18}. 
	
	Denote by $V$ a vector in $E$. We have the formal Taylor series expansion as follows
	\begin{equation}\label{eq:taylorG}\mathcal{G}(V)=\cD\mathcal{G}(V)+\frac{1}{2}\cD^2\mathcal{G}(V,V)+\frac{1}{6}\cD^3\mathcal{G}(V,V,V)+\cdots\end{equation}
	Here $\cD\mathcal{G}(V)=LV$, where $L$ is the linearised operator (at $0$). Its kernel $\mathcal{K}=\ker L$ forms the space of Jacobi fields, which correspond to infinitesimal deformations through critical points of $\mathcal{F}$. We are interested in whether such deformations are obstructed (non-integrable). 
	
	In general $\mathcal{D}^k\mathcal{G}$ is a symmetric $k$-linear map, and one might expect that $\mathcal{G}$ has an $m$th order obstruction so long as $\mathcal{D}^k\mathcal{G} =0$ for $k<m$, and $\|\mathcal{D}^m\mathcal{G} (V)\| \geq \delta \|V\|^m$. This is essentially true if $\mathcal{K}$ is not a proper subspace, but otherwise there will be more complicated conditions for an $m$th order obstruction, which we find inductively as follows.
	
	If $\mathcal{K}=0$, ellipticity shows that $L$ is invertible and hence $\mathcal{G}$ is obstructed at order 1 on all of $\mathcal{E}$. For example, a shrinking sphere $\mathbb{S}^k_{\sqrt{2k}}$ has no Jacobi fields modulo rotation, which gives an easy proof of rigidity. This is also the case for the product $ \mathbb{S}^{k_1}_{\sqrt{2k_1}} \times \mathbb{S}^{k_2}_{\sqrt{2k_2}}$ considered as a minimal hypersurface in $\mathbb{S}^{n+1}_{\sqrt{2n}}$, $n=k_1+k_2$, that is, using the area functional for $\mathcal{F}$. 
	
	If instead we consider $ \mathbb{S}^{k_1}_{\sqrt{2k_1}} \times \mathbb{S}^{k_2}_{\sqrt{2k_2}}$ as a self-shrinker, there do exist Jacobi fields other than rotations. Thus $\mathcal{K}\neq 0$ and we must look for higher order obstructions. To do so, we consider a complement $\mathcal{K}^\perp$ of $\mathcal{K}$ in $\mathcal{E}$ and decompose $V= U+h$, where $U\in\mathcal{K}$, $h\in\mathcal{K}^\perp$. 
	
	
	With respect to this decomposition, the Taylor expansion becomes
	\begin{equation}\label{eq:introtaylor}
		\begin{split}
			\mathcal{G}(V)&=\cD\mathcal{G}(h)+\frac{1}{2}\cD^2\mathcal{G}(U+h,U+h)+\frac{1}{6}\cD^3\mathcal{G}(U+h,U+h,U+h)+\cdots
			\\
			&= \cD\mathcal{G}(h)+\frac{1}{2}\cD^2\mathcal{G}(U,U)+\cD^2\mathcal{G}(U,h)+\frac{1}{2}\cD^2\mathcal{G}(h,h)+\cdots
		\end{split}
	\end{equation}
	Invertibility of $L$ on $\mathcal{K}^\perp$ shows that $U$ is first order and $h$ is (at least) second order in $\|V\|$, as expected, which justifies the grouping in the second line. Now $\mathcal{G}$ has an order 1 obstruction on $\mathcal{K}^\perp$, but even if $\mathcal{G}$ has an order 2 obstruction on $\mathcal{K}$, these do not automatically combine to give an order 2 obstruction on $\mathcal{E}$ since the $\frac{1}{2}\cD^2 \mathcal{G}(U,U)$ term could be cancelled by $\mathcal{D}\mathcal{G}(h)$. 
	
	For instance, if $\pi_\mathcal{K} \cD^2 \mathcal{G} (U,U) =0$, then we can cancel the corresponding term entirely by solving $LW=-\cD^2\varphi(U,U)$, and we expect $\frac{1}{2}W$ to be the principal part of $h$. (Strictly, we have used ellipticity of $L$ to identify the complement $\mathcal{C}$ of $\im \mathcal{G}$ with $\mathcal{K}$.) Indeed, if we decompose $h=h'+\frac{1}{2}W$, then
	\[
		\mathcal{G}(V)= \cD\mathcal{G}(h')+\frac{1}{2}\cD^2\mathcal{G}(U,W)+\frac{1}{6}\cD^3\mathcal{G}(U,U,U)+\cD^2\mathcal{G}(U,h')+\cdots
	\]
	Again, invertibility of $L$ on $\mathcal{K}^\perp$ shows that the residual $h'$ is of third order in $\|V\|$. Thus the first three terms are all third order, so an order 3 obstruction for $\mathcal{G}$ on $\mathcal{K}$, corresponding to the nondegeneracy of $\cD^3 \mathcal{G}(U,U,U)$ might cancel with $\cD^2\mathcal{G}(U,W)$ as well as $Lh'$. 
	
	Moreover, we have freedom to repeat the process by attempting to solve
	\begin{equation}\label{eq:LW'equation}
		LW'=-\left(\frac{1}{2}\cD^2\varphi(U,W)+\frac{1}{6}\cD^3\varphi(U,U,U)\right),
	\end{equation}
	and we would expect $W'$ to be the principal part of $h'$. This process continues until the projection to $\mathcal{K}$ is nonzero, which means that the corresponding equation is unsolvable. As in \cite{ELS18} this corresponds to the Jacobi field $U$ being non-integrable. 
	
	To obtain a quantitative rigidity theorem, we need these obstructions to occur in a uniform way, that is, at the same order for each $U$, and the projection of the right hand side to $\mathcal{K}$ should be bounded below; this quantifies the failure to solve the corresponding equation at that order. For $ \mathbb{S}^{k_1}_{\sqrt{2k_1}} \times \mathbb{S}^{k_2}_{\sqrt{2k_2}}$ as a self-shrinker, we find $\pi_\mathcal{K}\cD^2 \mathcal{G} (U,U) =0$ and $\|\pi_\mathcal{K}(\frac{1}{2}\cD^2\varphi(U,W)+\frac{1}{6}\cD^3\varphi(U,U,U))\| \geq C_3 \|U\|^3$ for all $U\in\mathcal{K}$, which gives an order 3 obstruction. This form of uniform obstruction was first observed in \cite{ELS18}, and the idea of directly estimating the size of the residuals at each order is inspired by \cite{CM19}. 
	
	
	\subsection{Technical considerations}
	Among the techniques we introduced in this paper, we want to emphasize two of them which play an important role. We believe that these techniques may also play an important role in other problems.

	The first is that we perform estimates in both H\"older spaces (Section \ref{sec:quant-rigidity}) and Sobolev spaces (Section \ref{sec:sobolev}). H\"{o}lder estimates would be enough to conclude rigidity, but Sobolev norms are natural for the \L ojasiewicz inequalities. Each also has its own technical advantages: Sobolev spaces behave well with our $L^2$-orthogonal decompositions, whilst H\"{o}lder norms behave well with products, which arise from Taylor expansion. It will be crucial that we first establish H\"{o}lder estimates, which allow us to iteratively estimate the Sobolev norms of higher degree terms using a refinement of the techniques in \cite{CM19} (see also \cite{Z20}). The subspace of Jacobi fields is an important bridge as it is a finite dimensional space, so the H\"older and Sobolev norms are equivalent.

	Second, we solve the linear elliptic equation $LW=-\mathcal{D}^2\mathcal{\varphi}(U,U)$ explicitly on $\mathbb{S}^{k_1}_{\sqrt{2k_1}} \times \mathbb{S}^{k_2}_{\sqrt{2k_2}}$, where $U$ is a non-rotational Jacobi field. This is challenging because there is no explicit formula for $L^{-1}$. In Section \ref{sec:jacobi} we show that $U = u\mathbf{H}$, where $u = \sum c_{ij} x_i y_j$ is a combination of products of coordinate functions on each factort; however, this would still mean solving a system of $O(k_1k_2)$ linear equations to find $W$. The resulting form of $W$ would also heavily complicate the later analysis. A key observation to keep the analysis tractable is that we may use the symmetry of each spherical factor to bring $u$ to the form $\sum a_i x_i y_i$. This actually yields a system of only 4 linear equations to solve (see Section \ref{sec:solW}). 
	
	We also choose not to compute $\frac{1}{2}\cD^2\varphi(U,W)+\frac{1}{6}\cD^3\varphi(U,U,U)$ completely explicitly, as suggested in \cite{ELS18}. Instead, to prove that the projection to $\mathcal{K}$ is nondegenerate, it is sufficient to take the inner product with $U$ itself, which is more tractable. 
	
	\subsection{Organisation of the paper}
	In Section \ref{sec:prelim}, we discuss some basic Riemannian geometry related to self-shrinkers and set out our notation and conventions. We then calculate general third variation formula in Section \ref{sec:variation}. In Section \ref{sec:jacobi}, we study the Jacobi fields on products of spheres. Analysis of higher order variations is presented in Section \ref{sec:analysis}. Using Taylor expansion, these are used in Section \ref{sec:quant-rigidity} to prove quantitative estimates, including the H\"{o}lder case of Theorem \ref{thm:quant-rigidity-intro}. The Sobolev case and the proof of Theorem \ref{thm:gradient-L} are then proven in Section \ref{sec:sobolev}. 
	Finally, Appendix \ref{sec:poly} contains the details of some calculations used in Section \ref{sec:analysis}.
	
	\subsection{Acknowledgements}
	
	The authors would like to thank Prof. Bill Minicozzi for his invaluable guidance and continued support, as well as Prof. Felix Schulze for insightful discussions about his work and helpful suggestions. 
	
	JZ was supported in part by the National Science Foundation under grant DMS-1802984 and the Australian Research Council under grant FL150100126.

	\section{Preliminaries}
	\label{sec:prelim}
	Let us recall some basic notions in differential geometry. We will follow the notations and conventions in \cite{CM19}. Suppose $\Sigma^n\subset\R^N$ is an $n$-dimensional submanifold. The second fundamental form is defined by $A(X,Y)=\nabla^\perp_X Y$. The mean curvature vector $\mathbf{H}$ is defined as minus the trace of the second fundamental form. Given a vector (field) $V$ on $\Sigma$, $A^V$ is defined to be the symmetric two tensor $\langle A,V\rangle$. We can also write these quantities in a local normal coordinate chart. Suppose $\{F_i\}$ is a tangent frame on $\Sigma$, then we denote by $A_{ij}=A(F_i,F_j)$, and $\mathbf{H}=- A_{ii}$. Here, and henceforth, we use the convention that repeated (tangent) indices are contracted via the metric. 
		
	Let $\phi$ be the normal vector field $\phi=\frac{x^\perp}{2}-\mathbf{H}$. A submanifold is a self-shrinker if and only if $\phi\equiv 0$. There are two important operators on self-shrinkers
	\[\cL=\Delta-\frac{1}{2}\nabla_{x^T},\]
	\[L=\cL+\frac{1}{2}+\sum_{k,l}\langle A_{kl},\cdot\rangle A_{kl}.\]
	The operator $\cL$ is defined on functions and tensors, while $L$ is defined on functions and tensors with values in the normal bundle.
	
	Now we specialise to submanifolds given by products of spheres. Let $\Sigma=\prod_{b=1}^B \mathbb{S}^{k_b}_{2k_b}\subset\prod_{b=1}^B\R^{k_b+1}\subset\R^N$. Let $x$ be the position vector field on $\Sigma$, and $g$ be the metric on $\Sigma$. On each sphere $\mathbb{S}^{k_b}_{\sqrt{2k_b}}$ we $N_b$ to be the outward unit normal, so that 
	\begin{equation}
		\begin{split}
			x= \sum_b \sqrt{2k_b}N_b,\ \ 
			A= - \sum_b \frac{g^b}{\sqrt{2k_b}}N_b,\ \ 
			\mathbf{H}= \sum_b\sqrt{\frac{k_b}{2}}N_b.		
		\end{split}
	\end{equation}
	Here $g^b$ is the metric on the factor $\mathbb{S}^{k_b}_{\sqrt{2k_b}}$. 
	
	As a consequence, on a product of spheres the $\mathcal{L}$ and $L$ operators are simply 
	\begin{equation}
		\cL=\Delta
	\end{equation} and 
	\begin{equation}\label{eq:L on product}
		L=\Delta+\frac{1}{2}+\frac{1}{2}\sum_{b=1}^B\Pi_{N_b},
	\end{equation}
	where $\Pi_{N_b}$ is the projection operator to $N_b$.

	\section{Jacobi fields}
	\label{sec:jacobi}

	In this section, we consider the self-shrinker $\Sigma^n =\prod_{b=1}^B \mathbb{S}^{k_b}_{\sqrt{2k_b}}\subset\prod_{b=1}^B\R^{k_b+1}\subset\R^N$ and analyse the constraints imposed by the first variation of $\phi$. That is, we determine the space $\mathcal{K} = \ker L$ of Jacobi fields on $\Sigma$. We also describe the space $\mathcal{K}_0$ of rotational Jacobi fields, and finally the space $\mathcal{K}_1$ of Jacobi fields orthogonal to rotations. 
	
	Recall that on $\mathbb{S}^{k}_{\sqrt{2k}} \subset \mathbb{R}^{k+1}$, the first non-zero eigenvalue of the Laplacian (on functions) is $\frac{1}{2}$, the space of $\frac{1}{2}$-eigenfunctions is spanned by $\theta_1, \cdots, \theta_{k+1}$, where $\theta_i = x_i |_{\mathbb{S}^{k}_{\sqrt{2k}}}$, and the next eigenvalue is strictly greater than $1$.
	
	As above let $\Sigma =\prod_{b=1}^B \mathbb{S}^{k_b}_{\sqrt{2k_b}}\subset\prod_{b=1}^B\R^{k_b+1}\subset\R^N$. We denote by $x^b_1,\cdots,x^b_{k_b+1}$ the standard coordinates on each $\R^{k_b+1}$, and $z_1,\cdots,z_{N-n-B}$ the standard coordinates on the rest of $\R^N$. As above we define $\theta^b_i=x^b_i |_\Sigma$. 
	
	The next proposition identifies all Jacobi fields on $\Sigma$.
	
	\begin{proposition}
		\label{prop:jacobi}
		Let $\Sigma^n =\prod_{b=1}^B \mathbb{S}^{k_b}_{\sqrt{2k_b}} \subset \mathbb{R}^N$. Suppose that $V\in \mathcal{K}$ is a Jacobi field, that is $LV=0$. Then for some $s^a_{ij}, t^\alpha_i \in \mathbb{R}$, we have
		\begin{equation}\label{eq:Jacobi V}
			V=\sum_{\substack{ b\neq c\\ a,i,j  }} s^a_{ij}\theta^b_i\theta^c_j N_a+\sum_{\alpha,i,b} t^\alpha_i\theta^b_i \partial_{z_\alpha}.
		\end{equation}
		
	\end{proposition}
	
	\begin{proof}
		Since the normal space to $\Sigma$ is spanned by the orthonormal parallel vector fields $\{N_b,\partial_{z_\alpha}\}$, we can always decompose a normal vector field $V$ into 
		\[V=\sum_b V^bN_b+\sum_\alpha V^\alpha \partial_{z_\alpha}.\]
		Then by \eqref{eq:L on product}, 
		\begin{equation}
			LV=\sum_b (\Delta+1)V^bN_b +\sum_{\alpha} (\Delta+\frac{1}{2})V^\alpha \partial_{z_\alpha}.
		\end{equation}
		Thus finding the Jacobi fields reduces to studying the eigenvalue problem for the Laplacian (on functions). 
		
		Note that $\Sigma$ is a product manifold, so its eigenvalues are sums of eigenvalues on each factor, and the eigenspaces are spanned by the products the eigenfunctions on each factor. Therefore we have the following:
		\begin{itemize}
			\item the $1$-eigenspace of $\Lap$ is spanned by $\{\theta^b_i\theta^c_j\}_{b\neq c}	$ ,
			\item the $\frac{1}{2}$-eigenspace of $\Lap$ is spanned by $\{\theta^b_i\}$.
		\end{itemize}
		This completes the proof. 
	\end{proof}
	
	Next we discuss the Jacobi fields generated by rotations. Every rotation $\mathcal{R} \in \mathrm{SO}(N)$ is generated by a vector field on $\mathbb{R}^N$. Its projection to the normal bundle of $\Sigma$ defines a Jacobi field $V_\mathcal{R}$. We define $\mathcal{K}_0:= \{V_\mathcal{R} | \mathcal{R}\in \mathrm{SO}(N)\} \subset \mathcal{K}$ to be the space of rotational Jacobi fields, or (infinitesimal) rotations for short. 
	
	\begin{lemma}
		Let $\Sigma^n =\prod_{b=1}^B \mathbb{S}^{k_b}(\sqrt{2k_b}) \subset \mathbb{R}^N$. Suppose that $V\in \mathcal{K}_0$ is a rotational Jacobi field. Then for some $s_{ij}, t^\alpha_i \in \mathbb{R}$, we have
		\begin{equation}\label{eq:Jacobi rot V}
			V=\sum_{\substack{b\neq c\\ i,j }} s_{ij} \theta^b_i\theta^c_j(\frac{1}{\sqrt{2k_c}}N_b-\frac{1}{\sqrt{2k_b}}N_c) +\sum_{\alpha,i,b} t^\alpha_i\theta^b_i \partial_{z_\alpha}.
		\end{equation}
		
	\end{lemma}
	\begin{proof}
		The group $\mathrm{SO}(N)$ is generated by rotations of 2 coordinates at a time. For two coordinate functions $x,y$ on $\R^N$, the rotation in the $xy$ plane is generated by the vector field $x \pr_y - y\pr_x$. Let $V_{xy} \in \mathcal{K}_0$ be the corresponding rotational Jacobi field. 
		
		Note that $\langle \partial_{x^b_i},N_b\rangle=\frac{1}{\sqrt{2k_b}}\theta^b_i$ and $\langle \partial_{x^b_i},N_a\rangle=0$ when $a\neq b$. 
		We now have different cases based on which subspace of $\R^N = \prod_{b=1}^B \mathbb{R}^{k_b+1} \times \mathbb{R}^{N-n-B}$ the coordinates $x,y$ correspond to:
		\begin{itemize}
			\item $x=x^b_i$, $y=x^b_j$ both belong to the same $\R^{k_b+1}$. Then $V_{xy}=\frac{1}{\sqrt{2k_b}}(\theta^b_i\theta^b_j-\theta^b_j\theta^b_i)N_b=0$;
			\item $x=z_\alpha$, $y=z_\beta$ both belong to the complementary subspace $\R^{N-n-B}$. Then we again have $V_{xy}=0$;
			\item $x=x^b_i$ and $y=z_\alpha$. Then $V_{xy}=\frac{1}{\sqrt{2k_b}}z_\alpha\theta^b_iN_b-\theta^b_i \partial_{z_\alpha}=-\theta^b_i \partial_{z_\alpha}$;
			\item $x=x^b_i$ and $y=x^c_j$ while $b\neq c$. Then the rotation generates the Jacobi field $\theta^b_i\theta^c_j(\frac{1}{\sqrt{2k_c}}N_b-\frac{1}{\sqrt{2k_b}}N_c)$. 
		\end{itemize}
		This completes the proof.
	\end{proof}
	
	Modding out by the rotational Jacobi fields $\theta^b_i \pr_{z_\alpha}$ immediately gives:
	
	\begin{proposition}
		Let $\Sigma^n =\prod_{b=1}^B \mathbb{S}^{k_b}_{\sqrt{2k_b}}$ and define $\mathcal{K}_1$ to be the $L^2$-orthocomplement of $\mathcal{K}_0$ in $\mathcal{K}$, that is, the space of Jacobi fields orthogonal to rotations. If $V\in \mathcal{K}_1$, then for some $s^a_{ij}\in \mathbb{R}$ we have
		\begin{equation}
			V=\sum_{a,b\neq c,i,j}s^a_{ij}\theta^b_i\theta^c_j N_a,
		\end{equation}
	\end{proposition}
	
	(Note that this is not a sufficient condition; not all choices of $s^a_{ij}$ yield $V\in \mathcal{K}_1$.) 
	
	In the case that $\Sigma$ is a product of two spheres, we can further simplify the expression of Jacobi fields in $\mathcal{K}_1$ after modding out by the remaining rotation fields $\theta^1_i\theta^2_j(\frac{1}{\sqrt{2k_2}}N_1-\frac{1}{\sqrt{2k_1}}N_2)$:
	
	\begin{proposition}
		\label{prop:jacobi-prod-2}
		Suppose $\Sigma=\mathbb{S}^{k_1}_{\sqrt{2k_1}}\times\mathbb{S}^{k_2}_{\sqrt{2k_2}} \subset \R^N$. If $V\in \mathcal{K}_1$, then	for some $s_{ij}\in \mathbb{R}$, \begin{equation}
			V=\sum_{i,j}s_{ij}\theta^1_i\theta^2_j(\sqrt{2k_1}N_1+\sqrt{2k_2}N_2).
		\end{equation}
	\end{proposition}
	
	\section{General variation of $\phi$}
	\label{sec:variation}
	
	In this section we consider (normal) variations of a submanifold, and describe the first, second and third variation of geometric quantities, in particular for $A$ and $\phi$. 
	
	Let $\Sigma$ be a submanifold with a fixed immersion $F_0 :\Sigma^n \rightarrow \mathbb{R}^N$, and a one-parameter family of immersions $F: I\times \Sigma^n\rightarrow \mathbb{R}^N$ with $F(0,p)=F_0(p)$. Using $s$ for the coordinate on $I=(-\epsilon,\epsilon)$, and subscripts to denote differentiation with respect to $s$. If $p_i$ are local coordinates on $\Sigma$, we get the tangent frame $F_i = F_*(\frac{\pr}{\pr p^i})$. 
	
	All geometric quantities such as $\Pi, g,A$ should be considered as functions of $s,p$, given by the value of each quantity at $F(s,p)$ on the submanifold defined by $X(s,\cdot)$. For instance, the metric $g_{ij}(s,p)$ is given by $g_{ij}=\langle F_i, F_j\rangle$. Recall $\Pi$ is the projection to the normal bundle. Also recall our convention that repeated lower indices are contracted via the (inverse) metric $g^{ij}$, although we will raise indices when it suits the exposition. 
	
	\subsection{First and second variation}
	
	\begin{proposition}[\cite{CM19}]
		\label{prop:1stvar}
		At $s=0$, let $F_s=V=V^\perp$; then we have
		
		\begin{alignat}{3}
			\label{eq:dPi}
			&\Pi_s(W) &&= -\Pi(\nabla_{W^T} V)  - F_j g^{ij}\langle \Pi(\nabla_{F_i} F_s), W\rangle,\\
			\label{eq:dg}
			&(g_{ij})_s &&= -2A^V_{ij}, (g^{ij})_s = 2g^{ik} A^V_{km} g^{mj}, \\
			\label{eq:dphi}
			&\phi_s = \mathcal{D}\mathcal{\varphi} (V)&&= LV - F_j g^{ij} \langle \nabla^\perp_{F_i} V,\phi\rangle.\end{alignat}
	\end{proposition}
	
	For the second variation
	
	\begin{lemma}
		\label{lem:d2Pi}
		At $s=0$, assume that $F_s^T= F_{ss}^T=0$; then the second variation $\Pi_{ss}$ acts by:

		\begin{alignat}{3}
					\label{eq:d2PiTan}
			&\Pi_{ss}(W^T)&& = -\Pi(\nabla_{W^T}F_{ss}) + 2\Pi(\nabla_{\nabla^T_{W^T} F_s} F_s) + 2 F_j g^{ij} \langle \Pi(\nabla_{F_i} F_s), \nabla_{W^T} F_s\rangle.\\
			\label{eq:d2PiNT}
			&\langle F_i, \Pi_{ss}\Pi(W)\rangle &&= 2\langle W,\Pi(\nabla_{\nabla^T_{F_i} F_s}F_s)\rangle - \langle W,\Pi(\nabla_{F_i} F_{ss})\rangle.\\
						\label{eq:d2PiNN}
			&\Pi \Pi_{ss}\Pi(W) &&= -2g^{ij}\langle W,\Pi(\nabla_{F_i} F_s)\rangle \Pi(\nabla_{F_j}F_s).\end{alignat}
		
	\end{lemma}
	
	\begin{proposition}
		At $s=0$, assume that $F_s = V = V^\perp$ and $F_{ss}=0$; then we have
		\begin{alignat}{3}
			\label{eq:d2g}
		&	(g_{ij})_{ss}&& = 2 g^{kl}A^V_{ik} A^V_{lj} + 2\langle \nabla^\perp_i V, \nabla_j^\perp V\rangle, \\
			\label{eq:d2ginv}
		&	(g^{ij})_{ss} &&= g^{i i_2} g^{jj_2}( 6 g^{kl} A^V_{i_2 k} A^V_{l j_2} - 2\langle \nabla^\perp_{i_2} V, \nabla_{j_2}^\perp V\rangle), 
\end{alignat}
		
	\end{proposition}
	
	\begin{corollary}
		\label{cor:d2phiN}
		Suppose that $\Sigma_0$ is a shrinker, that is, at $s=0$ we have $\phi \equiv 0$; then we have \begin{equation}
			\begin{split}
				\frac{1}{2}\phi_{ss} = \frac{1}{2}\mathcal{D}^2\mathcal{\varphi}(V,V) =&   A_{ij}A^V_{ik} A^V_{kj}  - A_{ij}\langle \nabla^\perp_i V, \nabla^\perp_j V\rangle \\&+ 2A^V_{ij} (\nabla^\perp\nabla^\perp V)_{ij} 
				+2\langle A_{ij}, \nabla_i^\perp V\rangle \nabla_j^\perp V 
				\\& - F_k \langle LV, \nabla^\perp_k V\rangle. 
			\end{split}
		\end{equation}
	\end{corollary}
	
	Later, we will need $\mathcal{D}^2\mathcal{\varphi}$ as a bilinear form; by polarisation it follows that:
	
	\begin{corollary}
		\label{cor:d2phiN-pol}
		Suppose that $\Sigma_0$ is a shrinker, then we have \begin{equation}
			\begin{split}
				\frac{1}{2}\mathcal{D}^2\mathcal{\varphi}(V,W) =&   A_{ij}A^V_{ik} A^W_{kj}  - A_{ij}\langle \nabla^\perp_i V, \nabla^\perp_j W\rangle \\&+ A^V_{ij} (\nabla^\perp\nabla^\perp W)_{ij}  + A^W_{ij} (\nabla^\perp\nabla^\perp V)_{ij}  \\&+\langle A_{ij}, \nabla_i^\perp V\rangle \nabla_j^\perp W+\langle A_{ij}, \nabla_i^\perp W\rangle \nabla_j^\perp V 
				\\& - \frac{1}{2}F_k (\langle LV, \nabla^\perp_k W\rangle+\langle LW, \nabla^\perp_k V\rangle). 
			\end{split}
		\end{equation}
	\end{corollary}

	\subsection{Third variation}
	
	\begin{lemma} At $s=0$, suppose $F^T_s = F^T_{ss}=F^T_{sss}=0$, then the third variation $\Pi_{sss}$ acts as 
		\begin{equation}
			\begin{split}
				\Pi_{sss}(W^T) =& -\nabla^\perp_{W^T} F_{sss} + 3\nabla^\perp_{\nabla^T_{W^T} F_{ss}} F_s + 3 F_jg^{ij}\langle \nabla^\perp_i F_s,\nabla^\perp_{W^T} F_{ss}\rangle + 3\nabla^\perp_{\nabla^T_{W^T}F_s} F_{ss} \\&- 6\nabla^\perp_{\nabla^T_{\nabla^T_{W^T}F_s}F_s}F_s 
				- 6 F_j g^{ij} \langle \nabla^\perp_i F_s, \nabla^\perp_{\nabla^T_{W^T}F_s}F_s \rangle
				\\&   -6g^{ij} F_j \langle \nabla^\perp_{W^T}F_s,  \nabla^\perp_{\nabla^T_i F_s} F_s\rangle + 6g^{ij} \langle \nabla^\perp_{W^T} F_s, \nabla^\perp_i F_s\rangle \nabla^\perp_j F_s \\& + 6g^{ij}F_j \langle \nabla^\perp_{W^T}F_s, \nabla^\perp_i F_{ss}\rangle,
			\end{split}
		\end{equation}

		\begin{equation}
			\begin{split}
				\langle F_i, \Pi_{sss} \Pi(W)\rangle =& -\langle W,\nabla^\perp_i F_{sss}\rangle +3\langle W, \nabla^\perp_{\nabla^T_{i}F_{ss}}F_s\rangle - 6\langle W, \nabla^\perp_{\nabla^T_{\nabla^T_i F_s}F_s}F_s\rangle
				\\& + 6 g^{jk} \langle \nabla^\perp_i F_s, \nabla^\perp_j F_s\rangle \langle \nabla^\perp_k F_s, W\rangle , 
			\end{split}
		\end{equation}

		\begin{equation}
			\begin{split}
				\Pi\Pi_{sss}\Pi(W) ={}& 3g^{ij}\langle W,\nabla^\perp_i F_s\rangle (-\nabla^\perp_j F_{ss} + 2\nabla^\perp_{\nabla^T_j F_s}F_s) \\&+ 3g^{ij}(2\langle W,\nabla^\perp_{\nabla^T_i F_s}F_s\rangle -\langle W,\nabla^\perp_i F_{ss}\rangle) \nabla^\perp_j F_s. 
			\end{split}
		\end{equation}
	\end{lemma}
	\begin{proof}
		Differentiating $\Pi^2=\Pi$ thrice, we have $\Pi_{sss} \Pi + 3\Pi_{ss} \Pi_s + 3\Pi_s \Pi_{ss} + \Pi \Pi_{sss}= \Pi_{sss}$ and hence $\Pi \Pi_{sss} \Pi = -3(\Pi_{ss} \Pi_s \Pi + \Pi_s \Pi_{ss}\Pi)$. 
		
		Similarly, differentiate $\Pi(X_i) =0$ thrice to get $\Pi_{sss}(F_i) + 3\Pi_s(F_{si}) + 3\Pi(F_{ssi}) + \Pi(F_{sssi})=0$ Finally, since $\Pi$ is a symmetric operator, so too are its derivatives, so for any $i$ we have $\langle F_i , \Pi_{sss}\Pi(W) \rangle = \langle \Pi_{sss}(F_i), \Pi(W)\rangle$. Substituting (\ref{eq:dPi}) and Lemma \ref{lem:d2Pi} into these relations gives the formulae above. \end{proof}
	
	\begin{lemma} 
		\label{lem:equiv-V}
		Let $V$ be a normal vector field on $\Sigma$, then 
		\begin{alignat*}{3}
		&\nabla^T_j V &&= -F_l g^{lk} \langle A_{jk}, V\rangle\\
		&\nabla^T_i\nabla^\perp_j V &&= -F_l g^{lk} \langle A_{ik}, \nabla^\perp_j V\rangle,\\
	&	\nabla_i\nabla^T_j  V&&=  -F_{il} g^{lk} \langle A_{jk}, V\rangle -  F_l g^{lk} \langle A_{jk}, \nabla^\perp_iV\rangle  -F_l g^{lk} \langle \nabla_i A_{jk}, V\rangle,\\
	&	\nabla_{\nabla^T_i V}V &&= -g^{lk}\langle A_{ik},V\rangle \nabla_l V,\\
	&	\nabla^T_{\nabla^T_i V}V &&= g^{lk}g^{mj} \langle A_{ik},V\rangle\langle A_{lm},V\rangle F_j,\\
	&	\nabla_{\nabla^T_{\nabla^T_iV}V}V &&= g^{lk}g^{mj} \langle A_{ik},V\rangle\langle A_{lm},V\rangle \nabla_j V.
		\end{alignat*}
	\end{lemma}
	\begin{proof}
		The key is to notice that
		\[\langle \nabla_i V,F_j\rangle=-\langle V,\nabla_i F_j\rangle=-\langle V,A_{ij}\rangle.\]
		This immediately gives the first two formulae. The third follows from differentiating the above, and the remainder follows after substituting the first two. 
	\end{proof}
	
	\begin{proposition}
		At $s=0$, assume that $F_s=V=V^\perp$ and $F_{ss}=F_{sss}=0$; then we have
		
		\begin{alignat*}{3} &(g_{ij})_{sss} &&= 0,\\
 &(g^{ik})_{sss} &&=24 A^V_{ij} A^V_{jl} A^V_{lk} - 12 A^V_{jk} \langle \nabla^\perp_i V, \nabla^\perp_j V\rangle - 12 A^V_{ij}\langle \nabla^\perp_j V, \nabla^\perp_k V\rangle.\
 \end{alignat*}
	\end{proposition}
	\begin{proof}
		Differentiate $g_{ij} = \langle F_i, F_j\rangle $ thrice and use that $F_{ss}=F_{sss}=0$. Then differentiate $g^{ik} g_{kj} = \delta^i_j$ thrice to find $(g^{ik})_{sss} g_{kj} = -3(g^{ij})_{ss} (g_{jk})_s -3 (g^{ij})_s (g_{jk})_{ss}$ and use the previous variation formulae (\ref{eq:dg}), (\ref{eq:d2g}) and (\ref{eq:d2ginv}). 
	\end{proof}

	\begin{proposition}
		\label{prop:d3phi}
		At $s=0$, assume that $F_s=V=V^\perp$ and $F_{ss}=F_{sss}=0$; then we have
		
		\begin{equation}
			\begin{split}
				\phi_{sss} =\mathcal{D}^3 \mathcal{\varphi}(V,V,V)=&  24 A^V_{ij} A^V_{jl} A^V_{lk} A_{ik} - 24 A_{ik} A^V_{jk} \langle \nabla^\perp_i V,\nabla^\perp_j V\rangle 
				\\&+ 6(3A^V_{ij}A^V_{jk} - \langle \nabla^\perp_i V,\nabla^\perp_k V\rangle)( (\nabla^\perp\nabla^\perp V)_{ik} - A_{il} A^V_{lk}) 
				\\&+12 A^V_{ik}  (\langle A_{ij}, \nabla^\perp_k V\rangle + \langle A_{jk}, \nabla^\perp_i V\rangle + \langle (\nabla A)_{jk,i},V\rangle) \nabla^\perp_j V 
				\\& -12\langle A_{ij},\nabla^\perp_i V\rangle A^V_{jk} \nabla^\perp_k V
				\\&- 12\langle \phi, \nabla_i^\perp V\rangle A^V_{ij} \nabla^\perp_j V -6 \langle \nabla^\perp_j \phi,V\rangle  A^V_{jk} \nabla^\perp_k V
				\\& -6\langle LV, \nabla^\perp_j V\rangle \nabla^\perp_j V  
				\\&  -3 F_j A^V_{ij} A^V_{ik} \langle \nabla_k^\perp V, \phi\rangle \\& +3F_j \langle \nabla^\perp_j V, \nabla^\perp_k V\rangle \langle \nabla^\perp_k V, \phi\rangle -6 F_j \langle \nabla^\perp_j V,\nabla^\perp_k V \rangle \langle \nabla_k^\perp \phi, V\rangle
				\\& - 6 F_j A^V_{jk} \langle LV, \nabla_k^\perp V\rangle -3 F_j \langle \mathcal{D}^2\mathcal{\varphi}(V,V) , \nabla^\perp_j V\rangle
			\end{split}
		\end{equation}

	\end{proposition}
	\begin{proof}
		Differentiating $-\mathbf{H} = g^{ij} \Pi(F_{ij})$ thrice and using that $F_{ss}=F_{sss}=0$ we have \begin{equation}\begin{split}-\mathbf{H}_{sss} ={}& (g^{ij})_{sss} \Pi(F_{ij}) +3 (g^{ij})_{ss} \Pi_s (F_{ij}) + 3(g^{ij})_s \Pi_{ss}(F_{ij})+ g^{ij} \Pi_{sss}(F_{ij}) \\& + 3(g^{ij})_{ss} \Pi(V_{ij}) + 6(g^{ij})_s \Pi_s (V_{ij}) + 3g^{ij} \Pi_{ss} (V_{ij}) .\end{split}\end{equation}
		
		Substituting the variation formulae for $g, \Pi$ above, we then have
		
		\begin{equation}
			\begin{split}
				-\mathbf{H}_{sss} ={}& 24 A^V_{ij} A^V_{jl} A^V_{lk} A_{ik} - 24 A_{ik} A^V_{jk} \langle \nabla^\perp_i V,\nabla^\perp_j V\rangle \\&+ 3(6A^V_{ij}A^V_{jk} - 2\langle \nabla^\perp_i V,\nabla^\perp_k V\rangle)(-\nabla^\perp_{F^T_{ik}}V - F_l g^{ml}\langle \nabla^\perp_m V, F_{ik}\rangle + \nabla^\perp_i\nabla^\perp_k V - A_{il}g^{ml} A^V_{mk}) \\&+12 A^V_{ik} \left(-\nabla^\perp_{\nabla^T_i\nabla_k V} V - F_l g^{jl} \langle \nabla^\perp_j V, \nabla^\perp_i \nabla_k V\rangle\right) \\&+ 6g^{ik}g^{jl}\left(-F_l\langle \nabla^\perp_j V,\nabla^\perp_{\nabla^T_{F^T_{ik}}V}V\rangle - F_l \langle \nabla^\perp_{\nabla^T_{\nabla^T_j V}V}V,F_{ik}\rangle \right)
				\\& +6g^{ik}g^{jl}\left(\langle \nabla^\perp_j V,F_{ik}\rangle \nabla^\perp_{\nabla^T_lV}V +\langle \nabla^\perp_{\nabla^T_l V}V, F_{ik}\rangle \nabla^\perp_j V\right)
				\\& +6g^{ik}g^{jl} g^{mq} F_l \langle F_{ik}, \nabla^\perp_m \rangle\langle \nabla_j^\perp V, \nabla^\perp_q V\rangle
				\\&+12A^V_{ik} \left(F_l g^{jl}\langle \nabla^\perp_j V, \nabla^\perp_{F^T_{ik}}V\rangle + F_l g^{jl} \langle A_{ik,} \nabla^\perp_{\nabla^T_j V}V\rangle - g^{jl}\langle A_{ik},\nabla^\perp_j V\rangle \nabla^\perp_l V\right)
				\\& - 6g^{ik}g^{jl} ( \langle A_{ij},\nabla^\perp_k V\rangle +\langle A_{jk},\nabla^\perp_i V\rangle + \langle \nabla_i A_{jk},V\rangle)\left(\nabla^\perp_{\nabla^T_l V}V + F_mg^{mq} \langle \nabla^\perp_q V, \nabla^\perp_lV\rangle\right)
				\\& + 6g^{ik}g^{jl} \left(-F_l \langle \nabla^\perp_{\nabla^T_jV}V,A_{im}\rangle g^{mq}A^V_{kq} + F_l \langle \nabla^\perp_{\nabla^T_j V} V,\nabla^\perp_i \nabla^\perp_k V\rangle\right)
				\\& + 6g^{ik}g^{jl} \left(-\langle \nabla^\perp_i\nabla^\perp_k V, \nabla^\perp_j V\rangle \nabla^\perp_l V + \langle A_{im}\nabla^\perp_j V\rangle \nabla^\perp_l V g^{mq} A^V_{kq}\right).
			\end{split}
		\end{equation}
		
		Using Lemma \ref{lem:equiv-V} to simplify the derivatives of $V$, collecting terms and using the Codazzi equation then gives
		
		\begin{equation}
			\begin{split}
				-\mathbf{H}_{sss} ={}& 24 A^V_{ij} A^V_{jl} A^V_{lk} A_{ik} - 24 A_{ik} A^V_{jk} \langle \nabla^\perp_i V,\nabla^\perp_j V\rangle 
				\\&+ 6(3A^V_{ij}A^V_{jk} - \langle \nabla^\perp_i V,\nabla^\perp_k V\rangle)( (\nabla^\perp\nabla^\perp V)_{ik} - A_{il} A^V_{lk}) 
				\\&+12 A^V_{ik}  (\langle A_{ij}, \nabla^\perp_k V\rangle + \langle A_{jk}, \nabla^\perp_i V\rangle + \langle (\nabla A)_{jk,i},V\rangle) \nabla^\perp_j V 
				\\&+ 12\langle H, \nabla_i^\perp V\rangle A^V_{ij} \nabla^\perp_j V
				\\&-6A^V_{ik} \langle A_{ik},\nabla^\perp_j V\rangle \nabla^\perp_j V
				\\& -6 ( 2\langle A_{ij},\nabla^\perp_i V\rangle -\langle \nabla^\perp_j H,V\rangle ) A^V_{jk} \nabla^\perp_k V
				\\& -6\langle \Lap^\perp V, \nabla^\perp_j V\rangle \nabla^\perp_j V 
				\\& -6F_j A^V_{il} A^V_{lk} \langle \nabla^\perp_j V , A_{ik} \rangle + 6 F_j \langle \nabla^\perp_j V, A_{ik}\rangle\langle \nabla^\perp_i V, \nabla^\perp_k V\rangle 
				\\& -12 F_j A^V_{ik} \langle (\nabla^\perp \nabla^\perp V)_{ik}, \nabla^\perp_j V\rangle
				\\& +6 F_j A^V_{ij} A^V_{ik} \langle \nabla^\perp_k V ,H\rangle  -6 F_j \langle \nabla^\perp_j V, \nabla^\perp_l V\rangle \langle \nabla^\perp_l V, H\rangle
				\\& -12F_j \langle \nabla^\perp_i V, A_{ik}\rangle \langle \nabla^\perp_j V,\nabla^\perp_k V\rangle + 12 F_j \nabla^\perp_j V,\nabla^\perp_k V\rangle \langle \nabla^\perp_k H, V\rangle 
				\\& -6 F_j A^V_{jl} A^V_{ik} \langle \nabla^\perp_l V , A_{ik}\rangle -6 F_j A^V_{jl} \langle \nabla^\perp_l V, \Lap^\perp V\rangle.
			\end{split}
		\end{equation}
		
		Similarly, we may compute $(x^\perp)_{sss} = (\Pi(F))_{sss}$, 
		
		\begin{equation}
			\begin{split}
				(x^\perp)_{sss} ={}& -6\langle V,\nabla^\perp_iV\rangle\nabla^\perp_i V - 6 \langle x, F_i\rangle \langle A_{ij},V\rangle\langle A_{jk},V\rangle\nabla^\perp_k V \\&- 6  \langle x^\perp,\nabla^\perp_i V\rangle \langle A_{ij},V\rangle\nabla^\perp_j V- 6\langle x^\perp,\nabla^\perp_j V\rangle \langle A_{ij},V\rangle \nabla^\perp_i V
				\\& +6 \langle \nabla^\perp_j V,\nabla^\perp_i V\rangle\langle x,F_j\rangle\nabla^\perp_i V
				\\& 
				+6 g^{ij}\langle x,F_l\rangle\langle \nabla^\perp_l V,\nabla^\perp_i V\rangle\langle A_{kj}, V\rangle F_k
				-6g^{ij}\langle x,\nabla^\perp_l V\rangle \langle A_{il},V\rangle \langle A_{kj},V\rangle F_k
				\\&
				+6\langle x,F_i\rangle \langle \nabla^\perp_j V,\nabla^\perp_k V\rangle \langle A_{ij},V\rangle F_k
				+6g^{ij}\langle x,\nabla^\perp_i V\rangle \langle \nabla^\perp_j V,\nabla^\perp_k V\rangle F_k
				\\&
				-6g^{ij}\langle V,\nabla^\perp_k V\rangle \langle A_{ik},V\rangle F_j. 
			\end{split}
		\end{equation}

		Using the definitions of $\phi$, $L$, as well as the formula for $\mathcal{D}^2\mathcal{\varphi}(V,V)$ and that $\nabla_k^\perp \phi = -\frac{1}{2}\langle x,X_j\rangle A_{jk}  -\nabla_k \mathbf{H}$ gives the result for $\phi_{sss}$.

	\end{proof}
	
	\subsection{Specialising to products of spheres}
	
	\begin{theorem}
		\label{thm:normal-var}
		Let $\Sigma =\prod_b \mathbb{S}^{k_b}_{\sqrt{2k_b}}$, and consider suppose $V,W$ are normal vector fields on $\Sigma$ of the form $V = \sum_b u^b N_b$, $W=\sum_b w^b N_b$. Then we have the variation formulae 
		
		\begin{multline}
			\label{eq:d2phiN}
				\Pi(\mathcal{D}^2\mathcal{\varphi}(V,V)) = \sum_{a,b} \Bigg(  - \frac{(u^b)^2}{\sqrt{2k_b}} N_b +  \sqrt{\frac{2}{k_b}} |\nabla^b u^a|^2 N_b \\ - 2\sqrt{\frac{2}{k_b}} u^a\Lap^a u^b N_b - 2\sqrt{\frac{2}{k_a}}\langle \nabla^a u^a, \nabla^a u^b\rangle N_b \Bigg),
		\end{multline}
		\begin{multline}
			\label{eq:d2phiN-pol}
				\Pi(\mathcal{D}^2\mathcal{\varphi}(V,W)) = \sum_{a,b} \Bigg(  - \frac{u^b w^b }{\sqrt{2k_b}} N_b + \sqrt{\frac{2}{k_b}} \langle \nabla^b u^a, \nabla^b w^a\rangle N_b  - \sqrt{\frac{2}{k_b}} (w^a \Lap^a u^b + u^a \Lap^a w^b) N_b  \\ - \sqrt{\frac{2}{k_a}} (\langle \nabla^a u^b , \nabla^a w^a\rangle + \langle \nabla^a w^b, \nabla^a u^a\rangle) N_b \Bigg),
		\end{multline}
		\begin{equation}
			\label{eq:d3phiN}
			\begin{split} 
				\Pi(\mathcal{D}^3\mathcal{\varphi}(V,V,V)) = \sum_{a,b} \Bigg( & -6 u^a \langle \nabla u^a,\nabla u^b\rangle N_b + 3 \frac{(u_b)^3N_b}{2k_b} - 18 \langle \nabla^b u^a , \nabla^b u^a\rangle  \frac{u^b N_b}{2k_b} \\&+18 \frac{(u^b)^2}{2k_b} \Lap^b u^a N_a  -6(\nabla^2 u^a)(\nabla u^b, \nabla u^b) N_a
				\\& +36 \frac{ u^b }{2k_b} \langle \nabla^b u^b, \nabla^b u^a\rangle N_a-6 \langle \nabla u^b ,\nabla  u^a\rangle \Lap u^b N_a \Bigg). 
			\end{split}
		\end{equation}
	\end{theorem}
	\begin{proof}
		Recall that on the shrinker $\Sigma$ we have $x=\sum_b \sqrt{2k_b}N_b$, and $A=-\sum_b\frac{g^b}{\sqrt{2k_b}}N_b$, hence \[\nabla^\perp_i V= \sum_b u^b_i N_b,\ \ A^V_{ij}=-\sum_b \frac{1}{\sqrt{2k_b}}u^b g^b_{ij}.\] The result follows straightforwardly after substituting into Corollaries \ref{cor:d2phiN}, \ref{cor:d2phiN-pol} and Proposition \ref{prop:d3phi}; we omit the calculations for brevity.  
	\end{proof}
	
	\begin{remark}
		\label{rmk:normalvariation}
		The normal variation formulae of Theorem \ref{thm:normal-var} were sufficient. In fact, the tangent parts at order $k$ vanish if the variations of all orders $j\leq k-1$ vanish. See Propositions \ref{prop:d2PhiK1} and \ref{prop:thirdvar} to follow. 
		
		Further, note that we have only considered normal fields orthogonal to the $\pr_{z_\alpha}$ directions in the above. This will be sufficient because all Jacobi fields orthogonal to rotation have this property, and as a consequence, the choice of $W$ in Section \ref{sec:solW} is also orthogonal to the $\pr_{z_\alpha}$ directions. 
	\end{remark}
	
	\section{Higher order variations on Clifford shrinkers}
	\label{sec:analysis}
	
	In this section we consider the case $\Sigma = \mathbb{S}^{k_1}_{\sqrt{2k_1}} \times \mathbb{S}^{k_2}_{\sqrt{2k_2}}$ and analyse the constraints given by the second and third variations of $\phi$. Specifically, we show in Proposition \ref{prop:d2PhiK1} that the second variation is orthogonal to the Jacobi part $U$. Formally, this constrains the (principal) non-Jacobi part $W$ to be a unique vector field depending on $U$, which is described in Section \ref{sec:solW}. Finally in Section \ref{sec:ord3obs} we are able to show that the third variation, including contribution from $W$, has a definite sign and so provides an obstruction. 
	
	In this section only, we write $x_i, y_j$ in place of $x^1_i, x^2_j$ and (by slight abuse of notation) their restrictions $\theta^1_i , \theta^2_j$ respectively. For convenience we denote $r_b = \sqrt{2k_b}$. Recall $\mathcal{K}_0$ denotes the space of rotation Jacobi fields in $L^2$, and $\mathcal{K}_1$ be the space of Jacobi fields orthogonal to rotations in $L^2$. 
	
	\begin{lemma}
		\label{lem:simpU}
		Suppose that $U\in\mathcal{K}_1$ is a Jacobi field after rotation. Then there are rotations of $\mathbb{R}^{k_1+1}$ and $\mathbb{R}^{k_2+1}$ such that $U = u^b N_b$, where $u^b = r_b u$ and $u = \sum_i a_i x_i y_i$ for some $a_i$. (In particular $a_i=0$ for $i> \min(k_1,k_2)+1$.)
	\end{lemma}
	\begin{proof}
		By Proposition \ref{prop:jacobi-prod-2} the statement holds with $u = \sum c_{ij} x_i y_j$. Applying singular value decomposition to the matrix $(c_{ij})$ gives the desired form after rotations of $\mathbb{R}^{k_b+1}$. 
	\end{proof}
	
	For the remainder of this section, fix $U = u^b N_b$ as above with $u=\sum_i a_i x_i y_i$. Define $v_1 = \sum_i a_i^2 x_i^2$ and $v_2 = \sum_i a_i^2 y_i^2$. Then $u^2 = \sum_{ij} a_ia_j x_ix_jy_iy_j$ and \[|\nabla^b u|^2 = v_{\bar{b}} - r_b^{-2} u^2,\] where $\bar{b}$ denotes the opposite index to $b=1,2$. Note that $\Lap^b u = -\frac{1}{2}u$ and so $\Lap u =-u$.  Then by the second variation formula (\ref{eq:d2phiN}) we have
	
	\begin{equation}
			\langle \mathcal{D}^2\mathcal{\varphi}(U,U), N_b\rangle = - r_b u^2  +  \sum_{a} \left( 2 r_a^2 r_b^{-1} (v_{\bar{b}} - r_b^{-2} u^2)  +2 r_a u^2  - 4 r_b (v_{\bar{a}} - r_a^{-2} u^2) \right).
	\end{equation}
	
	\begin{proposition}
		\label{prop:d2PhiK1}
		Let $U\in \mathcal{K}_1$. Then $\mathcal{D}^2\mathcal{\varphi}(U,U) = \Pi(\mathcal{D}^2\mathcal{\varphi}(U,U))$ and $\pi_{\mathcal{K}}(\mathcal{D}^2\mathcal{\varphi}(U,U))=0$.
	\end{proposition}
	\begin{proof}
		Since $LU=0$, Corollary \ref{cor:d2phiN} implies that the tangent part of the second variation $\mathcal{D}^2\mathcal{\varphi}(U,U)$ vanishes. Now after rotations of $\mathbb{R}^{k_1+1}$ and $\mathbb{R}^{k_2+1}$, which fix $\Sigma$, we may assume that $U$ is as in Lemma \ref{lem:simpU}. 
		
		Consider $V = x_iy_j N_b$. It follows from the above calculations that $\int_\Sigma \langle \mathcal{D}^2\mathcal{\varphi}(U,U) ,V\rangle$ is the integral of a polynomial of odd degree in both $\mathbb{R}^{k_b+1}$, $b=1,2$, and hence is equal to 0. 
		
		Since $\langle U,\pr_{z_\alpha}\rangle =0$, by (\ref{eq:d2phiN}) we have that $\langle \mathcal{D}^2\mathcal{\varphi}(U,U) , \pr_{z_\alpha}\rangle=0$. By Proposition \ref{prop:jacobi} this implies that $\mathcal{D}^2\mathcal{\varphi}(U,U)$ is $L^2$-orthogonal to every element $\mathcal{K}$, which completes the proof. 
	\end{proof}
	
	We now proceed to analyse the third variation. 
	
	\subsection{Spherical integrals}
	
	We will reduce all integrals to the following list. Let \[\beta_k(b)=\beta_k(b_1,\cdots, b_{k+1}) = 2\frac{\prod \Gamma(\frac{b_i+1}{2})}{\Gamma(\sum \frac{b_i+1}{2})}.\] We understand that \[\beta_k(b_1,\cdots, b_j) = \beta_k(b_1,\cdots, b_j, 0 , \cdots,0) \] if $j<k$. 
	Then the integral $\int_{\mathbb{S}^{k}(r)} \prod x_i^{b_i}$ is 0 if any $b_i$ are odd, and $\beta_{k}(b) r^{k+|b|}$ otherwise.
	
	\begin{lemma}
		\label{lem:sphint}
		We have 
		\begin{alignat}{3} &\int_\Sigma v_1 &&= r_1^{k_1+2} r_2^{k_2}  \beta_{k_1}(2) \beta_{k_2}(0) \sum a_i^2 
				\\&&&=r_1^2|\Sigma| \frac{1}{k_1+1} \sum a_i^2\nn
				,\\
		& \int_\Sigma v_2 &&= r_1^{k_1} r_2^{k_2+2} \beta_{k_1}(0) \beta_{k_2}(2) \sum a_i^2  
				\\&&&=r_2^2|\Sigma| \frac{1}{k_2+1} \sum a_i^2\nn
				,\\
		& \int_\Sigma v_1^2&& = r_1^{k_1+4} r_2^{k_2} \beta_{k_2}(0) \left( \beta_{k_1}(4) \sum a_i^4 + 2\beta_{k_1}(2,2) \sum_{i<j} a_i^2 a_j^2\right) 
				\\&&&=r_1^4 |\Sigma| \frac{1}{(k_1+1)(k_1+3)} \left(3\sum a_i^4 + 2 \sum_{i<j} a_i^2 a_j^2\right) \nn
				,\\
		&\int_\Sigma v_2^2 &&= r_1^{k_1} r_2^{k_2+4} \beta_{k_1}(0) \left( \beta_{k_2}(4) \sum a_i^4 + 2\beta_{k_2}(2,2) \sum_{i<j} a_i^2 a_j^2\right) 
				\\&&&= r_2^4 |\Sigma| \frac{1}{(k_2+1)(k_2+3)} \left(3\sum a_i^4 + 2 \sum_{i<j} a_i^2 a_j^2\right) \nn
				,\\
		& \int_\Sigma v_1 v_2&& = r_1^{k_1+2} r_2^{k_2+2} \beta_{k_1}(2)\beta_{k_2}(2) \left(\sum a_i^2\right)^2 
				\\&&&= r_1^2 r_2^2 |\Sigma| \frac{1}{(k_1+1)(k_2+1)} \left(\sum a_i^2\right)^2\nn
				,\\
		& \int_\Sigma u^2 &&= r_1^{k_1+2} r_2^{k_2+2} \beta_{k_1}(2)\beta_{k_2}(2) \left(\sum a_i^2\right) 
				\\&&&= r_1^2 r_2^2 |\Sigma| \frac{1}{(k_1+1)(k_2+1)} \left(\sum a_i^2\right)\nn
				,\\
		& \int_\Sigma u^2 v_1 && = r_1^{k_1+4} r_2^{k_2+2} \beta_{k_2}(2) \left( \beta_{k_1}(4)\sum a_i^4 + 2\beta_{k_1}(2,2) \sum_{i<j} a_i^2 a_j^2\right) 
				\\&&&= r_1^4 r_2^2 |\Sigma| \frac{1}{(k_1+1)(k_1+3)(k_2+1)}\left( 3 \sum a_i^4 + 2 \sum_{i<j} a_i^2 a_j^2\right)\nn
				,\\
		&\int_\Sigma u^2 v_2  &&= r_1^{k_1+2} r_2^{k_2+4} \beta_{k_1}(2) \left( \beta_{k_2}(4)\sum a_i^4 + 2\beta_{k_2}(2,2) \sum_{i<j} a_i^2 a_j^2\right) 
				\\&&&= r_1^2 r_2^4 |\Sigma| \frac{1}{(k_1+1)(k_2+1)(k_2+3)}\left( 3 \sum a_i^4 + 2 \sum_{i<j} a_i^2 a_j^2\right)\nn
				,\\
		& \int_\Sigma u^4  &&= r_1^{k_1+4} r_2^{k_2+4}  \left( \beta_{k_1}(4)\beta_{k_2}(4)\sum a_i^4 + 6 \beta_{k_1}(2,2)\beta_{k_2}(2,2) \sum_{i<j} a_i^2 a_j^2\right)
				\\&&&= r_1^4 r_2^4 |\Sigma| \frac{1}{(k_1+1)(k_1+3)(k_2+1)(k_2+3)}\left( 9 \sum a_i^4 + 6 \sum_{i<j} a_i^2 a_j^2\right)\nn
				.\end{alignat}
	\end{lemma}
	\begin{proof}
		The lemma follows by degree counting, and calculating that \[\frac{\beta_k(2)}{\beta_k(0)} = \frac{\Gamma(3/2)}{\Gamma(1/2)}\frac{\Gamma(\frac{k+1}{2})}{\Gamma(1+\frac{k+1}{2})} = \frac{1}{k+1} ,\] \[\frac{\beta_k(4)}{\beta_k(0)} = \frac{\Gamma(5/2)}{\Gamma(1/2)}\frac{\Gamma(\frac{k+1}{2})}{\Gamma(2+\frac{k+1}{2})} = \frac{3}{(k+1)(k+3)} ,\] \[\frac{\beta_k(2,2)}{\beta_k(0)} = \left(\frac{\Gamma(3/2)}{\Gamma(1/2)}\right)^2 \frac{\Gamma(\frac{k+1}{2})}{\Gamma(2+\frac{k+1}{2})} =  \frac{1}{(k+1)(k+3)} .\] Note that integrating $u^2$ against polynomials of even degree is equivalent to integrating $a_i^2 x_i^2 y_i^2$ against the same polynomials. 
	\end{proof}

	\subsection{Solving for $W$}
	
	\label{sec:solW}
	
	To cancel the second variation due to $U$, we find the unique $W \in \mathcal{K}^\perp$ solving \[LW = - \mathcal{D}^2\mathcal{\varphi}(U,U).\] Note that this is possible since by Proposition \ref{prop:d2PhiK1} the right hand side is indeed in $\mathcal{K}^\perp = \im L$. 
	
	By (\ref{eq:d2phiN}), if $U$ is a Jacobi field orthogonal to rotations then $\langle \mathcal{D}^2\mathcal{\varphi}(U,U), \pr_{z_\alpha}\rangle=0$, so $W$ will take the form $W=\sum_b w^b N_b$. Therefore $LW = (\Lap+1)w^b N_b$. We calculate that \[\Lap u^2 = -2u^2 + 2|\nabla u|^2 = -2(1+r_1^{-2}+r_2^{-2})u^2 + 2 v_1 +2 v_2.\] Also $\Lap v_b = -(1+2r_b^{-2}) v_b + 2\sum a_i^2$. These calculations imply:
	
	\begin{lemma}\label{lem:LW}
		Suppose
		$w = Au^2 + Bv_1 + C v_2 + D\sum a_i^2,$ then \[(\Lap+1)w = -Au^2(1+2r_1^{-2} + 2r_2^{-2}) + 2(A-B r_1^{-2}) v_1 + 2(A-C r_2^{-2}) v_2 + (2B+2C+D)\sum a_i^2.\] 
	\end{lemma}
	Note that $W=\sum_b w^b N_b$ is indeed orthogonal to $\mathcal{K}$ if the $w^b$ take the form in Lemma \ref{lem:LW}. 
	
	From the second variation formula we would like to solve \begin{equation}
		\begin{split}
			-(\Lap+1)w^b = {}& \langle \mathcal{D}^2\mathcal{\varphi}(U,U), N_b\rangle \\=&  u^2\left(-r_b - \sum_a 2r_a^2 r_b^{-3} + 2r_a + 4r_b r_a^{-2}\right)   + \sum_a \left(2 r_a^2 r_b^{-1} v_{\bar{b}} - 4 r_b v_{\bar{a}}\right).
		\end{split}
	\end{equation}
	
	Introduce the notation $M_d = \sum_b r_b^d$, e.g. $M_0=2$ and $M_2 = 2n$. 
	
	Comparing coefficients, we find the solution
	
	\[w^b = A^b u^2 + B^b v_1 + C^b v_2 + D^b \sum a_i^2,\]
	
	where \begin{alignat*}{3}&A^b = \frac{ - 2M_2r_b^{-3} + 2M_1 + (4M_{-2}-1) r_b}{1+2M_{-2}},\\
	 &B^1 = r_1^2 (A^1  -2r_1), && B^2 = r_1^2 (A^2 + M_2 r_2^{-1}-2 r_2),\\
	& C^1 = r_2^2 (A^1 + M_2 r_1^{-1} -2 r_1), && C^2 = r_2^2( A^2  -2r_2),\\
	&D^b = -2B^b-2C^b.\end{alignat*}
	
	\subsection{Second variation cross term}
	
	Rewriting the second variation formula (\ref{eq:d2phiN-pol}) in terms of $u$ and $r_b$, it follows that
	
	\begin{equation}
		\begin{split}
			\langle \mathcal{D}^2\mathcal{\varphi}(U,W),U\rangle=  -\sum_b r_b w^b u^2 + \sum_{a,b} \Bigg( & 2 r_a u\langle \nabla^b u,\nabla^b w^a\rangle - 2u (r_b w^a \Lap^a u + r_a u \Lap^a w^b)\\&   - 2r_b u (r_b r_a^{-1}\langle \nabla^a u , \nabla^a w^a\rangle + \langle \nabla^a w^b, \nabla^a u\rangle) \Bigg). 
		\end{split}
	\end{equation}
	
	We can avoid any derivatives on $w$ using the following integration by parts identities:
	
	$\int_\Sigma\langle \nabla^b u, \nabla^b w^a\rangle u = \int_\Sigma w^a(\frac{1}{2}u^2 -|\nabla^b u|^2)$, 
	
	$\int_\Sigma u^2 \Lap^a w^b = \int_\Sigma w^b(-u^2 + 2|\nabla^a u|^2)$,
	
	$\int_\Sigma\langle \nabla^a u,\nabla^a w^b\rangle u = \int_\Sigma w^b (\frac{1}{2}u^2 - |\nabla^a u|^2)$.
	
	Using these implies: 
	
	\begin{equation}
		\begin{split}
			\langle \mathcal{D}^2\mathcal{\varphi}(U,W),U\rangle_{L^2}= - \sum_b r_b \int_\Sigma w^b u^2 + \sum_{a,b} \Bigg( &2 r_a \int_\Sigma w^a (\frac{1}{2}u^2 r_b^0 - |\nabla^b u|^2) + r_b \int_\Sigma w^a u^2 \\& - 2 r_a \int_\Sigma w^b (-u^2 + 2|\nabla^a u|^2)  \\& -2r_b^2 r_a^{-1} \int_\Sigma w^a (\frac{1}{2}u^2 - |\nabla^a u|^2)\\& - 2r_b \int_\Sigma w^b (\frac{1}{2}u^2 r_a^0 - |\nabla^a u|^2)  \Bigg). 
		\end{split}
	\end{equation}
	
	Here we have left in factors $r_a^0$ to help guide the reader. Substituting $|\nabla^b u|^2 = v_{\bar{b}} - r_b^{-2} u^2$ and swapping the roles of some indices $a,b$, we have 
	
	\begin{equation}
		\begin{split}
			\langle \mathcal{D}^2\mathcal{\varphi}(U,W),U\rangle_{L^2}= -\sum_b r_b \int_\Sigma w^b u^2 +\sum_{a,b} \Bigg( & 2 r_b \int_\Sigma w^b (\frac{1}{2}u^2 r_a^0 - v_{\bar{a}} + r_a^{-2} u^2) + r_a \int_\Sigma w^b u^2 \\& - 2 r_a \int_\Sigma w^b (-u^2 + 2v_{\bar{a}} -2r_a^{-2} u^2)  \\&-2r_a^2 r_b^{-1} \int_\Sigma w^b (\frac{1}{2}u^2 - v_{\bar{b}} +r_b^{-2}u^2 ) \\&- 2r_b \int_\Sigma w^b (\frac{1}{2}u^2 r_a^0 - v_{\bar{a}} + r_a^{-2}u^2) \Bigg).
		\end{split}
	\end{equation}

	Collecting terms and performing some sums over $a$ gives finally:
	
	\begin{equation}
		\label{eq:D2phiUWU}
		\begin{split}
			\langle \mathcal{D}^2\mathcal{\varphi}(U,W),U\rangle_{L^2}=&\sum_{b}   (3M_1  + 4M_{-1} - M_2  r_b^{-1} -2M_2 r_b^{-3} -r_b )\int_\Sigma w^b u^2    \\&  + \sum_b 2M_2 r_b^{-1} \int_\Sigma w^b v_{\bar{b}} - \sum_{a,b} 4r_a \int_\Sigma w^b v_{\bar{a}}. \end{split}
	\end{equation}

	\subsection{Third variation}
	
	Rewriting the third variation formula (\ref{eq:d3phiN}) in terms of $u$ and $r_b$, it follows that
	
	\begin{equation}
		\begin{split} \langle \mathcal{D}^3\mathcal{\varphi}(U,U,U), U\rangle = \sum_{a,b} \Bigg(& -6 r_a^2 r_b^2 u^2 |\nabla u|^2
			+3r_b^2 u^4 - 18r_a^2 |\nabla^b u|^2 u^2 + 18 r_a^2 u^3 \Lap^b u \\& - 6r_a^2r_b^2 u \nabla^2 u (\nabla u,\nabla u) + 36 r_a^2 u^2 |\nabla^b u|^2 - 6 r_a^2 r_b^2 u\Lap u |\nabla u|^2 \Bigg) . 
		\end{split}
	\end{equation}
	
	Now we use the integration by parts identities $\int_\Sigma u^2 |\nabla u|^2 = - \frac{1}{3} \int_\Sigma u^3 \Lap u = \frac{1}{3}\int_\Sigma u^4$, similarly $\int_\Sigma u^2 |\nabla^b u|^2 = \frac{1}{6}\int_\Sigma u^4$, and \[\int_\Sigma u \nabla^2 u(\nabla u,\nabla u) = \int_\Sigma u u_{ij} u_i u_j = -\frac{1}{2} \int_\Sigma |\nabla u|^4 + u\Lap u |\nabla u|^2 = -\frac{1}{2} \int_\Sigma |\nabla u|^4 - \frac{1}{3} u^4.\] 
	
	Collecting terms and performing the sums as in the previous section then gives: 
	
	\begin{equation}
		\label{eq:D3phiUUUU}
		\langle \mathcal{D}^3\mathcal{\varphi}(U,U,U), U\rangle_{L^2} = ( -9M_2 - M_2^2)\int_\Sigma u^4 + 3M_2^2 \int_\Sigma |\nabla u|^4. 
	\end{equation}

	\subsection{Third order obstruction}
	\label{sec:ord3obs}

	\begin{proposition}
		\label{prop:thirdvar}
		There exists $\delta>0$ such that for any $U\in \mathcal{K}_1$, if $W\in \mathcal{K}^\perp$ is the solution of $LW = -\mathcal{D}^2\mathcal{\varphi}(U,U)$ then
		
		\[ \mathcal{D}^3\mathcal{\varphi} +3 \mathcal{D}^2 \mathcal{\varphi} (U,W) = \Pi(\mathcal{D}^3\mathcal{\varphi} +3 \mathcal{D}^2 \mathcal{\varphi} (U,W)),\]
		
		\begin{equation}\|\pi_{\mathcal{K}_1}(\mathcal{D}^3\mathcal{\varphi}(U,U,U) + 3\mathcal{D}^2\mathcal{\varphi}(U,W))\|_{L^2} \geq \delta \|U\|_{L^2}^3.\end{equation}
		
	\end{proposition}
	\begin{proof}
		By choice of $W$, it follows from Corollary \ref{cor:d2phiN-pol} and Proposition \ref{prop:d3phi} that the tangent part of $\frac{1}{6}\mathcal{D}^3\mathcal{\varphi}(U,U,U)  + \frac{1}{2} \mathcal{D}^2 \mathcal{\varphi} (U,W)$ is zero.

		To prove the estimate on the projection to $\mathcal{K}_1$, it is enough to prove that \begin{equation}\label{eq:thirdvar}\langle \mathcal{D}^3\mathcal{\varphi}(U,U,U) + 3\mathcal{D}^2\mathcal{\varphi}(U,W), U\rangle_{L^2} \geq \delta \|U\|_{L^2}^4.\end{equation} 
		
		After a rotation which fixes $\Sigma$, we may assume $U = u^b N_b$ with $u^b =r_bu$ and $u =\sum_i a_i x_iy_i$ as before. Then $\|U\|_{L^2}^2 = c \sum_i a_i^2$, where $c=2 r_1^2 r_2^2 |\Sigma| \frac{1}{(k_1+1)(k_2+1)}$. We compute the left hand side of (\ref{eq:thirdvar}) by expanding equation (\ref{eq:D2phiUWU}) using the solution of $W$ given in Section \ref{sec:solW}, and equation (\ref{eq:D3phiUUUU}) using  \[|\nabla u|^4 = (v_1+v_2 -(r_1^{-2}+r_2^{-2})u^2)^2.\] This reduces each term to one of the spherical integrals computed in Lemma \ref{lem:sphint}.
		
		Finally, using that $2\sum_{i<j} a_i^2 a_j^2 = \left(\sum_i a_i^2\right)^2 - \sum a_i^4$, we find
		
		\[\langle \mathcal{D}^3\mathcal{\varphi}(U,U,U) + 3\mathcal{D}^2\mathcal{\varphi}(U,W)), U\rangle_{L^2} = |\Sigma|  \left( Q_4 \sum_i a_i^4 + Q_2 \left(\sum_i a_i^2\right)^2\right) ,\] where $Q_2,Q_4$ are certain rational functions of $r_1,r_2$. We also introduce the quantity \[Q_0 := \frac{Q_4}{1+2r_1^2} + Q_2.\] We will need some positivity properties of these functions; the explicit form of $Q_0, Q_2, Q_4$ and the proofs of these properties are left to the appendix. In particular, we have:
		
		Claim 1: $Q_4\geq 0$ for any $r_1,r_2 \geq 0$. 
		
		Claim 2: $Q_0\geq \delta(r_1,r_2)>0$ for any $r_1,r_2 \geq \sqrt{2}$. 
		
		Given the claims, note that by the power means inequality we have $\sum a_i^4 \geq \frac{1}{\#\{a_i \neq 0 \}}\left(\sum_i a_i^2\right)^2.$ But there were at most $\min(k_1,k_2)+1$ nonzero $a_i$. Without loss of generality assume that $k_1\leq k_2$. Then by Claim 1, we have \[Q_4 \sum a_i^4 + Q_2 \left(\sum_i a_i^2\right)^2 \geq \frac{Q_4}{1+k_1}\left(\sum_i a_i^2\right)^2 + Q_2 (\sum a_i^2)^2 = Q_0 \left(\sum_i a_i^2\right)^2.\] The desired inequality (\ref{eq:thirdvar}) now follows directly from Claim 2.  
	\end{proof}
	
	\section{Quantitative rigidity}
	\label{sec:quant-rigidity}
	
	The main goal of this section is to prove the H\"{o}lder case of the quantitative rigidity Theorem \ref{thm:quant-rigidity-intro}. It will follow easily after we prove our main estimate for normal fields orthogonal to rotations:
	
	\begin{theorem}
		\label{thm:quant-rigidity}
		Let $\Sigma = \mathbb{S}^{k_1}_{\sqrt{2k_1}} \times \mathbb{S}^{k_2}_{\sqrt{2k_2}}$. There exists $\epsilon_0 >0$ such that if $V \in \mathcal{K}_0^\perp$ and $\|V\|_{C^{2,\alpha}} \leq \epsilon_0$, then 	\begin{equation}
			\|V\|_{C^{2,\alpha}}^3
			\leq \|\phi_V\|_{C^{0,\alpha}}.
		\end{equation}
	\end{theorem}
	
	In this section we use Taylor expansion about $\Sigma$ to quantify the variational analysis of Sections \ref{sec:jacobi} and \ref{sec:analysis}. Specifically, first order expansion will show that the Jacobi part $U$ of $V$ is dominant. That is, the non-Jacobi part $h$ must be higher order; second order expansion then shows that $h$ must be $\frac{1}{2}W$ up to an even higher order term $h'$. Third order expansion finally controls $U$, which allows us to finally conclude that $\|V\|_{C^{2,\alpha}} \leq \|\phi_V\|_{C^{0,\alpha}}$. 
	
	Throughout this section, $\Sigma= \mathbb{S}^{k_1}_{\sqrt{2k_1}} \times \mathbb{S}^{k_2}_{\sqrt{2k_2}}$ and $V$ is a normal vector field on $\Sigma$ with $\|V\|_{C^{2,\alpha}} <\epsilon$. To begin the analysis, note that shrinker quantity $\phi$ defines an analytic functional $\mathcal{\varphi}$ so that $\phi_V = \mathcal{\varphi}(x, V,\nabla V, \nabla^2 V)$. In particular \begin{equation}\label{eq:phiC0a} \|\phi_V\|_{C^{0,\alpha}} \leq C \|V\|_{C^{2,\alpha}} \leq C\epsilon ,\end{equation} so lower powers of $V,\phi_V$ are dominant for small enough $\epsilon$. 
	
	We will repeatedly use the following trick:
	
	\begin{lemma}
		\label{lem:projection}
		Let $V$ be a normal vector field on $\Sigma$ and $\mathcal{K}$ a finite dimensional subspace of $L^2$. Then 
		\begin{equation}
			\|\pi_{\mathcal{K}} (V)\|_{C^{2,\alpha}} \leq C  \|\pi_{\mathcal{K}} (V)\|_{L^2} \leq C'\|V\|_{L^2}\leq  C'' \|V\|_{C^{0,\alpha}}.
		\end{equation}
	\end{lemma}
	\begin{proof}
		The first inequality is the equivalence of norms on the finite dimensional space $\mathcal{K}$, the second is because $\pi_{\mathcal{K}}$ is the $L^2$-projection and the last inequality is just integration. 
	\end{proof}
	
	We also need the following elliptic estimate:
	
	\begin{lemma}[Schauder estimate for $L$]
		\label{lem:schauder}
		There exists $C$ so that if $V$ is normal and $U=\pi_\mathcal{K}(V)$, where $\mathcal{K}=\ker L$ is the space of Jacobi fields, then
		\begin{equation}
			\|V-U\|_{C^{2,\alpha}}\leq C\|LV\|_{C^{0,\alpha}}.
		\end{equation}
	\end{lemma}

	We proceed to apply Taylor expansion to $\phi_{sV}$ about $s=0$.

	\subsection{First order expansion}
	\label{sec:ord1}
	
	In this section we do not need to assume $V$ is orthogonal to rotations. We decompose $V= U+h$ where $U =\pi_\mathcal{K}(V)$ and $\mathcal{K}\subset L^2$ is the space of Jacobi fields, which is finite dimensional. Then $h\in \mathcal{K}^\perp$. 
	
	First, we have the following Taylor expansion estimate:
	
	\begin{lemma}
		\label{lem:ord1}
		\begin{equation}
			\|\phi_V - Lh\|_{C^{0,\alpha}} \leq C \|V\|_{C^{2,\alpha}}^2.
		\end{equation}
	\end{lemma}
	\begin{proof}
		Define $\phi(s) = \phi_{sV} = \mathcal{\varphi}(x,sV,s\nabla V,s\nabla^2 V)$. Then $\phi^{(k)}(s) = (\mathcal{D}^k\mathcal{\varphi})_{sV} (V,\cdots,V)$, and it follows that $\|\phi^{(k)}(s)\|_{C^{0,\alpha}} \leq C \|V\|_{C^{2,\alpha}}^k$. Moreover note that $\phi(0)=0, \phi'(0) = LV= Lh$. Applying Taylor expansion about $s=0$ (using the integral form of the remainder) yields the result. 
	\end{proof}
	
	This first order expansion allows us to show that $h$ is of higher order than $U$: 
	
	\begin{proposition}
		\label{prop:ord1-est}
		There exist $C$ and $\epsilon$ so that if $V$ is a normal field with $\|V\|_{C^{2,\alpha}} \leq \epsilon$, and $h$ is as above, then 
		\begin{equation}
			\|V\|_{C^{2,\alpha}} \leq C(\|V\|_{C^{0,\alpha}} + \|\phi_V\|_{C^{0,\alpha}}),
		\end{equation}
		\begin{equation}
			\|h\|_{C^{2,\alpha}} \leq C(\|V\|_{C^{0,\alpha}}^2 + \|\phi_V\|_{C^{0,\alpha}}). 
		\end{equation}
	\end{proposition}
	\begin{proof}
		By the triangle inequality 
		$
		\|V\|_{C^{2,\alpha}} \leq \|U\|_{C^{2,\alpha}} + \|h\|_{C^{2,\alpha}},
		$
		and since $U =\pi_\mathcal{K} (V)$ by Lemma \ref{lem:projection} we have $\|U\|_{C^{2,\alpha}} \leq C \|V\|_{C^{0,\alpha}}$. Now by Lemmas \ref{lem:schauder} and \ref{lem:ord1}, we have \[\|h\|_{C^{2,\alpha}} \leq C\|Lh\|_{C^{0,\alpha}} \leq C(\|V\|_{C^{2,\alpha}}^2 + \|\phi_V\|_{C^{0,\alpha}}).\]
		
		Absorbing the higher power of $\|V\|_{C^{2,\alpha}}$ gives the estimate for $V$ and then absorbing the higher power of $\|\phi_V\|_{C^{0,\alpha}}$ gives the estimate for $h$. 
	\end{proof}

	\subsection{Second order expansion}
	
	Henceforth, we assume for simplicity that $V$ is orthogonal to rotations, so that $U\in \mathcal{K}_1$. 
	
	\begin{lemma}
		\label{lem:ord2}
		\begin{equation}
			\label{eq:ord2}
			\|\phi_V - Lh - \frac{1}{2}\mathcal{D}^2\mathcal{\varphi}(U,U) \|_{C^{0,\alpha}} \leq C( \|U\|_{C^{2,\alpha}} \|h\|_{C^{2,\alpha}} + \|h\|_{C^{2,\alpha}}^2 + \|V\|_{C^{2,\alpha}}^3).
		\end{equation}
	\end{lemma}
	\begin{proof}
		Proceed by Taylor expansion as in Lemma \ref{lem:ord1}, noting that \[\phi''(0) = \mathcal{D}^2\mathcal{\varphi}(V,V) = \mathcal{D}^2\mathcal{\varphi}(U,U) + 2\mathcal{D}^2\mathcal{\varphi}(U,h) +\mathcal{D}^2\mathcal{\varphi}(h,h).\] 
	\end{proof}
	
	Now we decompose $h = \frac{1}{2}W+ h'$, where $W\in \mathcal{K}_1^\perp$ satisfies $LW = -\mathcal{D}^2\mathcal{\varphi}(U,U)$ as in Section \ref{sec:analysis}. The second order expansion shows that $h'$ is of still higher order:
	
	\begin{proposition}
		There exist $C$ and $\epsilon$ so that if $V$ is a normal field, orthogonal to rotations, with $\|V\|_{C^{2,\alpha}} \leq \epsilon$, and $h'$ is as above, then 
		\begin{equation}
			\|h'\|_{C^{2,\alpha}}
			\leq 
			C(\|V\|^3_{C^{0,\alpha}}
			+\|\phi_V\|_{C^{0,\alpha}}).
		\end{equation}
	\end{proposition}
	\begin{proof}
		By choice of $W$, the left hand side of (\ref{eq:ord2}) becomes $\|\phi_V - Lh'\|_{C^{0,\alpha}}$. So using Lemma \ref{lem:schauder} and the previous estimates, we have
		\begin{alignat*}{3}
			\|h'\|_{C^{2,\alpha}} &\leq{}&& \|Lh'\|_{C^{0,\alpha}} \leq \|\phi_V\|_{C^{0,\alpha}}+ \|\phi_V - Lh'\|_{C^{0,\alpha}} \\&\leq {}&&
			\|\phi_V\|_{C^{0,\alpha}}+ C\|V\|_{C^{0,\alpha}}(\|V\|_{C^{0,\alpha}}^2 + \|\phi_V\|_{C^{0,\alpha}}) \\&&&+C\left((\|V\|_{C^{0,\alpha}}^4 + \|\phi_V\|_{C^{0,\alpha}}^2) + (\|V\|_{C^{0,\alpha}}^3 + \|\phi_V\|_{C^{0,\alpha}}^3)\right).
		\end{alignat*}
		
		Estimating mixed terms using Young's inequality and absorbing higher powers gives the result.
	\end{proof}
	\subsection{Third order expansion}

	\begin{lemma}
		\label{lem:ord3}
		\begin{equation}
			\label{eq:ord3}
			\begin{split}
				\|\phi_V - Lh' +& \frac{1}{6}( \mathcal{D}^3\mathcal{\varphi}(U,U,U) + 3 \mathcal{D}^2\mathcal{\varphi}(U,W))) \|_{C^{0,\alpha}} \\&\leq C( \|U\|_{C^{2,\alpha}} \|h'\|_{C^{2,\alpha}}  + \|h'\|_{C^{2,\alpha}}^2 + \|W\|_{C^{2,\alpha}}^2 + \|h'\|_{C^{2,\alpha}} \|W\|_{C^{2,\alpha}})
				\\&\quad+C( \|U\|_{C^{2,\alpha}}^2 \|h\|_{C^{2,\alpha}}+\|U\|_{C^{2,\alpha}} \|h\|_{C^{2,\alpha}}^2 + \|h\|_{C^{2,\alpha}}^3 + \|V\|_{C^{2,\alpha}}^4).
			\end{split}
		\end{equation}
	\end{lemma}
	\begin{proof}
		Again proceed by Taylor expansion as in Lemma \ref{lem:ord1}, expanding out \[\phi'''(0) - \mathcal{D}^3\mathcal{\varphi}(U,U,U) = 3\mathcal{D}^3\mathcal{\varphi}(U,U,h) + 3\mathcal{D}^3\mathcal{\varphi}(U,h,h) + \mathcal{D}^3\mathcal{\varphi}(h,h,h) ,\] 
		\[ \phi''(0) -  \mathcal{D}^2\mathcal{\varphi}(U,U) - 2 \mathcal{D}^2\mathcal{\varphi}(U,W) = 2 \mathcal{D}^2\mathcal{\varphi}(U,h' ) +  \mathcal{D}^2\mathcal{\varphi}(h',h') +  \mathcal{D}^2\mathcal{\varphi}(W,W) + 2 \mathcal{D}^2\mathcal{\varphi}(h',W),\]  and using the choice of $W$ so that $Lh+\frac{1}{2}\mathcal{D}\mathcal{\varphi}(U,U) = Lh'$. 
	\end{proof}

	\begin{proof}[Proof of Theorem \ref{thm:quant-rigidity}]
		
		By elliptic theory (or integrating by parts) we see $\pi_{\mathcal{K}_1}(Lh')=0$. Then by Proposition \ref{prop:thirdvar} and the equivalence of norms on the finite dimensional space $\mathcal{K}_1$, it follows that
		\begin{equation}
			\|U\|_{C^{2,\alpha}}^3\leq C \|\pi_{\mathcal{K}_1}(-Lh'+\frac{1}{6}( \mathcal{D}^3\mathcal{\varphi}(U,U,U) + 3 \mathcal{D}^2\mathcal{\varphi}(U,W)))\|_{C^{2,\alpha}}.
		\end{equation}

		Applying Lemma \ref{lem:projection} to $ -Lh' + \frac{1}{6}( \mathcal{D}^3\mathcal{\varphi}(U,U,U) + 3 \mathcal{D}^2\mathcal{\varphi}(U,W))$, we then have
		\begin{equation} 
			\label{eq:U3}
			\begin{split}
				\frac{1}{C}\|U\|_{C^{2,\alpha}}^3 &\leq \| -Lh' + \frac{1}{6}( \mathcal{D}^3\mathcal{\varphi}(U,U,U) + 3 \mathcal{D}^2\mathcal{\varphi}(U,W))\|_{C^{0,\alpha}} \\&\leq \|\phi_V\|_{C^{0,\alpha}} + \| \phi_V -Lh' + \frac{1}{6}( \mathcal{D}^3\mathcal{\varphi}(U,U,U) + 3 \mathcal{D}^2\mathcal{\varphi}(U,W))\|_{C^{0,\alpha}}.
			\end{split} 
		\end{equation}
		
		We now bound the right hand side of (\ref{eq:ord3}), using the estimates from the first and second order expansions, as well as the Schauder estimate, to find
		\begin{equation}
			\|W\|_{C^{2,\alpha}} \leq C\|LW\|_{C^{0,\alpha}} \leq C' \|U\|_{C^{2,\alpha}}^2 \leq C''\|V\|^2_{C^{0,\alpha}}.
		\end{equation} 
		
		This gives
		
		\begin{alignat*}{3}&\|U\|_{C^{2,\alpha}} \|h'\|_{C^{2,\alpha}} && \leq C \|V\|_{C^{0,\alpha}} ( \|V\|^3_{C^{0,\alpha}}
		+\|\phi_V\|_{C^{0,\alpha}}),\\
		&\|h'\|^2_{C^{2,\alpha}} &&\leq C (\|V\|^6_{C^{0,\alpha}}
		+\|\phi_V\|^2_{C^{0,\alpha}}),\\
		&\|W\|_{C^{2,\alpha}}^2 &&\leq C\|V\|^4_{C^{0,\alpha}},\\
		&\|h'\|_{C^{2,\alpha}} \|W\|_{C^{2,\alpha}}&&\leq C\|V\|^2_{C^{0,\alpha}} (\|V\|^3_{C^{0,\alpha}}
		+\|\phi_V\|_{C^{0,\alpha}}),\\
		& \|U\|_{C^{2,\alpha}}^2 \|h\|_{C^{2,\alpha}} &&\leq C \|V\|_{C^{0,\alpha}}^2 (\|V\|^2_{C^{0,\alpha}}
		+\|\phi_V\|_{C^{0,\alpha}}),\\
		& \|U\|_{C^{2,\alpha}} \|h\|^2_{C^{2,\alpha}} &&\leq C \|V\|_{C^{0,\alpha}} (\|V\|^4_{C^{0,\alpha}}
		+\|\phi_V\|^2_{C^{0,\alpha}}),\\
		&\|h\|^3_{C^{2,\alpha}}&& \leq C(\|V\|^6_{C^{0,\alpha}}
		+\|\phi_V\|^3_{C^{0,\alpha}}),\\
		&\|V\|^4_{C^{2,\alpha}}&& \leq C(\|V\|^4_{C^{0,\alpha}}
		+\|\phi_V\|^4_{C^{0,\alpha}}).
		\end{alignat*}
		
		Using these to bound (\ref{eq:ord3}) and hence (\ref{eq:U3}), after absorbing higher powers we find that 
		
		\begin{equation}
			\|U\|_{C^{0,\alpha}}^3 \leq C(\|V\|_{C^{0,\alpha}}^4+\|\phi_V\|_{C^{0,\alpha}}). 
		\end{equation}
		
		Since $V = U+ \frac{1}{2}W + h'$, we then have 
		\begin{equation}
			\begin{split}
				\|V\|_{C^{2,\alpha}} & \leq  \|U\|_{C^{2,\alpha}} + \|W\|_{C^{2,\alpha}} + \|h'\|_{C^{2,\alpha}} \\& \leq C( \|V\|_{C^{0,\alpha}}^\frac{4}{3}+\|\phi_V\|^\frac{1}{3}_{C^{0,\alpha}} + \|V\|^2_{C^{0,\alpha}} +\|V\|^3_{C^{0,\alpha}} +\|\phi_V\|_{C^{0,\alpha}}).
			\end{split}
		\end{equation}
		
		Collecting higher powers of $\phi_V$ and absorbing all powers of $V$ into the left hand side then yields the desired estimate 
		\begin{equation}
			\|V\|_{C^{2,\alpha}} \leq C\|\phi_V\|^\frac{1}{3}_{C^{0,\alpha}}.
		\end{equation} 
		
	\end{proof}

	\subsection{Rotation and quantitative rigidity}
	
	We need the following rotation lemma:
	
	\begin{lemma}[c.f. Lemma 5.1 in \cite{ELS18}]\label{lem:modulo rotation}
		\label{lem:slice}
		Let $\Sigma^n \subset \mathbb{R}^N$ be a closed shrinker. For any $\epsilon>0$ there exists $\delta=\delta(\epsilon)>0$ such that if $\Sigma' \subset \mathbb{R}^N$ can be written as a graph of a vector field $V'$ on $\Sigma$ with $\|V'\|_{C^1}<\delta$, then there exist a rotation $\mathcal{R} \in \mathrm{SO}(N)$ and a $C^1$ vector field $V\in \mathcal{K}_0^\perp$ with $\|V\|_{C^1}<\epsilon$ such that $\mathcal{R}\cdot\Sigma'$ is the graph of $V$ over $\Sigma$.
	\end{lemma}

	\begin{proof}

		As in \cite{ELS18}, this is an application of the slice theorem for differentiable Lie group actions. Specifically, $G=\mathrm{SO}(N)$ acts on the space $\mathcal{M}$ of $C^1$-submanifolds of $\mathbb{R}^N$. The tangent space at $\Sigma$ is given by the space of $C^1$ sections of $N\Sigma$, and the tangent space to the orbit $G\cdot\Sigma$ is precisely the space $\mathcal{K}_0$ of rotational Jacobi fields. 
	\end{proof}
	
	(Note that when $\Sigma = \prod_{b=1}^B \mathbb{S}^{k_b}_{\sqrt{2k_b}}$, the stabiliser $G_\Sigma$ is the subgroup $\prod_{b=1}^B SO(k_b+1)$.) We also need the following lemma controlling graphs under rotation of the base:
	
	\begin{lemma}
		\label{lem:change-base}
		Let $\Sigma \subset \mathbb{R}^N$ be a closed submanifold. For any $\epsilon_0$ there exists $\epsilon'_0$ so that if $\Sigma' $ is the graph of $V'$ over $\Sigma$ and $\mathcal{R}\cdot \Sigma'$ is the graph of $V$ over $\Sigma$, with $\|V\|_{C^1}, \|V'\|_{C^{2,\alpha}} <\epsilon'_0$, then $\|V\|_{C^{2,\alpha}} <\epsilon_0$. 
	\end{lemma}
	\begin{proof}
		The $C^1$ bounds on both $V,V'$ control the size of the rotation $\mathcal{R}$. This allows us to transfer the $C^{2,\alpha}$ bounds to $V$. 
	\end{proof}
	
	We are now able to prove the quantitative rigidity theorem: 
	
	\begin{proof}[Proof of Theorem \ref{thm:quant-rigidity-intro}, H\"{o}lder estimate]
		Let $\epsilon_0$ be as in Theorem \ref{thm:quant-rigidity}. Take $\epsilon'_0$ as in Lemma \ref{lem:change-base} and $\delta=\delta(\epsilon'_0)$ as in Lemma \ref{lem:slice}. 
		
		By supposition $\Sigma'$ may be written as the normal graph of a vector field $V'$ over $\Sigma$ with $\|V'\|_{C^{2,\alpha}} \leq \epsilon$. Then Lemma \ref{lem:slice} gives a rotation $\mathcal{R}$ such that $\mathcal{R}\cdot \Sigma'$ is the graph of some $V\in \mathcal{K}_0^\perp$ with $\|V\|_{C^1} <\epsilon_1$, and hence $\|V\|_{C^{2,\alpha}} <\epsilon_0$. We conclude from Theorem \ref{thm:quant-rigidity} that \[\|V\|_{C^{2,\alpha}} \leq C\|\phi_V\|_{C^{0,\alpha}}^\frac{1}{3}.\] \end{proof}
	
	\section{Sobolev theory}\label{sec:sobolev}
	
	In this section we complete the proof of Theorem \ref{thm:quant-rigidity-intro} by Taylor expansion with Sobolev norms. The analysis formally mirrors the H\"{o}lder analysis in Section \ref{sec:quant-rigidity}, but the main difficulty is that Taylor expansion leads to higher powers of our quantities, so we need to iteratively control them in $W^{2,p}$ where $p$ increases with the order of expansion. It will thus be crucial that we have already proven $C^{2,\alpha}$ estimates in Section \ref{sec:quant-rigidity}, which will allow us to absorb terms of excess degree. At the final, obstructed order, we perform the analysis in $L^1$ to keep the degree below what we have already estimated. The main estimate is:
	
	\begin{theorem}
		\label{thm:quant-rigidity-L2}
		Let $\Sigma = \mathbb{S}^{k_1}_{\sqrt{2k_1}} \times \mathbb{S}^{k_2}_{\sqrt{2k_2}}$. There exists $\epsilon_0 >0$ such that if $V \in \mathcal{K}_0^\perp$ and $\|V\|_{C^{2,\alpha}} \leq \epsilon_0$, then 	\begin{equation}
			\|V\|_{L^2}^3
			\leq \|\phi_V\|_{L^2}.
		\end{equation}
	\end{theorem}
	
	We continue with the notation of Section \ref{sec:quant-rigidity}. Specifically, $\Sigma= \mathbb{S}^{k_1}_{\sqrt{2k_1}} \times \mathbb{S}^{k_2}_{\sqrt{2k_2}}$ and $V$ is a normal vector field on $\Sigma$ with $\|V\|_{C^{2,\alpha}} <\epsilon_0$. Additionally, we define $|V|_k := \sum_{j\leq k} |\nabla^j V|$ so that $\int_\Sigma |V|_k^p = \|V\|_{W^{k,p}}^p$. 
	
	We need the following elliptic estimate:
	
	\begin{lemma}[$W^{2,2}$ elliptic estimate for $L$]
		\label{lem:ellipticW22}
		There exists $C$ so that if $V$ is normal and $U=\pi_\mathcal{K}(V)$, where $\mathcal{K}=\ker L$ is the space of Jacobi fields, then
		\begin{equation}
			\|V-U\|_{W^{2,2}}\leq C\|LV\|_{L^2}.
		\end{equation}
	\end{lemma}
	
	\subsection{First order expansion}
	
	As above decompose $V= U+h$ where $U =\pi_\mathcal{K}(V)$ and $h\in \mathcal{K}^\perp$. By Taylor expansion as in Lemma \ref{lem:ord1}, we have $|\phi_V- Lh| \leq C|V|_2^2$. Integrating, we then have
	\begin{equation}
		\label{eq:ord1-L2}
		\|\phi_V - Lh\|_{L^2} \leq C \|V\|_{W^{2,4}}^2.
	\end{equation}
	
	\begin{proposition}
		\label{prop:ord1-est-L2}
		There exist $C$ and $\epsilon$ so that if $V$ is a normal field with $\|V\|_{C^{2,\alpha}} \leq \epsilon$, and $h$ is as above, then 
		\begin{equation}
			\|h\|_{W^{2,2}} \leq C(\|V\|_{L^2}^2 + \|\phi_V\|_{L^2}),
		\end{equation}
	\end{proposition}
	\begin{equation}
		\|V\|_{W^{2,4}} \leq C(\|V\|_{L^2} + \|\phi_V\|_{L^2}). 
	\end{equation}
	
	\begin{proof}
		By Lemma \ref{lem:ellipticW22} and (\ref{eq:ord1-L2}) we have \begin{equation}\label{eq:ord1-est-L2-1}\|h\|_{W^{2,2}} \leq C \|Lh\|_{L^2} \leq C'( \|\phi_V\|_{L^2} + \|V\|_{W^{2,4}}^2).\end{equation} 
		Now by the triangle inequality we have \begin{equation}|V|_2^4 \leq 2|V|_2^2(|U|_2^2 + |h|_2^2) \leq 4(|U|_2^4 + \epsilon_0^2 |h|_2^2),\end{equation} where we have used the absorbing inequality $2ab \leq \frac{1}{2}a^2 + 2b^2$ for the first term and the $C^2$ bound on $V$ for the second. By the equivalence of norms on the finite dimensional space $\mathcal{K}$,  \begin{equation}\label{eq:U2-est} \|U\|_{C^{2,\alpha}} \leq C\|U\|_{L^2} \leq C'\|V\|_{L^2}, \end{equation}
		and hence $\|U\|_{W^{2,4}} \leq C\|U\|_{L^2} \leq C\|V\|_{L^2}$. Therefore
		\begin{equation}\label{eq:ord1-est-L2-3} \|V\|_{W^{2,4}}^2 \leq C\|V\|_{L^2}^2 + 4\epsilon_0 \|h\|_{W^{2,2}}.\end{equation}
		Substituting (\ref{eq:ord1-est-L2-3}) into (\ref{eq:ord1-est-L2-1}) and absorbing the $\epsilon_0$ term gives the estimate for $\|h\|_{W^{2,2}}$ and consequently for $\|V\|_{W^{2,4}}$. 
	\end{proof}
	
	\subsection{Second order expansion}
	
	In the remainder of this section, we assume $V$ is orthogonal to rotations, so $U\in \mathcal{K}_1$. By (\ref{eq:phiC0a}) and the $C^{2,\alpha}$ estimate of Proposition \ref{prop:ord1-est}, for any $\epsilon_1>0$ we will have $\|h\|_{C^{2,\alpha}} <\epsilon_1$ so long as $\epsilon_0$ is small enough. 
	
	As before we decompose $h = \frac{1}{2}W+ h'$, where $W\in \mathcal{K}_1^\perp$ satisfies $LW = -\mathcal{D}^2\mathcal{\varphi}(U,U)$. By Taylor expansion as in Lemma \ref{lem:ord2} and integrating, we have
	\begin{equation}
		\label{eq:ord2-L2}
		\|\phi_V - Lh' \|_{L^2} \leq C( \| |U|_2^2 |h|_2^2 \|_{L^2} +\|h\|_{W^{2,4}}^2 + \|V\|_{W^{2,6}}^3).
	\end{equation}
	
	\begin{proposition}
		\label{prop:ord2-est-L2}
		There exist $C$ and $\epsilon_0$ so that if $V$ is a normal field, orthogonal to rotations, with $\|V\|_{C^{2,\alpha}} \leq \epsilon_0$, and $h'$ is as above, then 
		\begin{equation}
			\|h'\|_{W^{2,2}}
			\leq 
			C(\|V\|^3_{L^2}
			+\|\phi_V\|_{L^2}),
		\end{equation}
		
		\begin{equation}
			\|h\|_{W^{2,4}}
			\leq 
			C(\|V\|^2_{L^2}
			+\|\phi_V\|_{L^2}),
		\end{equation}
		
		\begin{equation}
			\|V\|_{W^{2,6}}
			\leq 
			C(\|V\|_{L^2}
			+\|\phi_V\|_{L^2}).
		\end{equation}
		
	\end{proposition}
	\begin{proof}
		By Lemma \ref{lem:ellipticW22} and (\ref{eq:ord2-L2}), we have \begin{equation}\label{eq:ord2-est-L2-1}\|h'\|_{W^{2,2}} \leq C \|Lh'\|_{L^2} \leq C'( \|\phi_V\|_{L^2} +\| |U|_2 |h|_2 \|_{L^2}+\|h\|_{W^{2,4}}^2 + \|V\|_{W^{2,6}}^3).\end{equation} 
		
		By (\ref{eq:U2-est}) we have \begin{equation} \label{eq:ord2-est-L2-2} \| |U|_2 |h|_2 \|_{L^2} \leq C \|V\|_{L^2} \|h\|_{W^{2,2}} \leq C\|V\|_{L^2}(\|V\|_{L^2}^2 + \|\phi_V\|_{L^2}).\end{equation} 
		
		For the remaining terms, we use absorbing tricks similar to before: 
		\begin{equation}
			\label{eq:ord2-est-L2-3}
			|h|_2^4 \leq 2|h|_2^2(\frac{1}{4}|W|_2^2 + |h'|_2^2) \leq |W|_2^4 + 4\epsilon_1^2 |h'|_2^2. 
		\end{equation}
		
		Now by Lemma \ref{lem:schauder} and (\ref{eq:U2-est}), we have 
		\begin{equation}
			\label{eq:W2-est}
			\|W\|_{C^{2,\alpha}} \leq C\|\mathcal{D}^2\mathcal{\varphi}(U,U)\|_{C^{0,\alpha}} \leq C' \|U\|_{C^{2,\alpha}}^2 \leq C'' \|V\|_{L^2}^2. 
		\end{equation}
		Integrating (\ref{eq:ord2-est-L2-3}) then gives
		\begin{equation}
			\label{eq:ord2-est-L2-6}
			\|h\|_{W^{2,4}}^2 \leq C \|V\|_{L^2}^4 + 4\epsilon_1^2  \|h'\|_{W^{2,2}}. 
		\end{equation}
		
		Similarly using Young's inequality and absorbing, we have \begin{equation}
			\label{eq:ord2-est-L2-7}
			|V|_2^6 \leq 8|V|_2^2(|U|_2^4 + |h|_2^4) \leq 2^{9/2} |U|_2^6 + 12\epsilon_0^2 |h|_2^4,
		\end{equation}
		and therefore 
		\begin{equation}
			\label{eq:ord2-est-L2-8}
			\|V\|_{W^{2,6}}^3 \leq C \|V\|_{L^2}^3 + 12 \epsilon_0^2 \|h\|_{W^{2,4}}^2 \leq C'\|V\|_{L^2}^3 + 48\epsilon_0^2\epsilon_1^2 \|h'\|_{W^{2,2}}.
		\end{equation}
		
		Substituting (\ref{eq:ord2-est-L2-2}), (\ref{eq:ord2-est-L2-6}) and (\ref{eq:ord2-est-L2-8}) into (\ref{eq:ord2-est-L2-1}), and absorbing the $\epsilon_0^2\epsilon_1^2 \|h'\|_{W^{2,2}}$ term into the left hand side gives the estimate for $\|h'\|_{W^{2,2}}$ (after keeping only dominant terms). The estimates for $h$ and $V$ then follow after substituting back into (\ref{eq:ord2-est-L2-6}) and (\ref{eq:ord2-est-L2-8}) respectively.
	\end{proof}
	
	\subsection{Third order expansion and distance \L ojasiewicz}
	
	Once more we proceed by Taylor expansion as in Lemma \ref{lem:ord3}, but we only integrate to $L^1$. This yields: 
	\begin{equation}
		\label{eq:ord3-L1}
		\begin{split}
			\|\phi_V - Lh' +& \frac{1}{6}( \mathcal{D}^3\mathcal{\varphi}(U,U,U) + 3 \mathcal{D}^2\mathcal{\varphi}(U,W))) \|_{L^1} \\&\leq C( \||U|_2 |h'|_2 \|_{L^1}  + \|h'\|_{W^{2,2}}^2 + \|W\|_{W^{2,2}}^2 + \||h'|_2|W|_2 \|_{L^1})
			\\&\quad+C( \||U|_2^2 |h|_2\|_{L^1}+\| |U|_2 |h|_2^2 \|_{L^1} + \|h\|_{W^{2,3}}^3 + \|V\|_{W^{2,4}}^4).
		\end{split}
	\end{equation}
	
	Before proving the $L^2$-quantitative rigidity Theorem \ref{thm:quant-rigidity-L2}, we need the following lemma:
	
	\begin{lemma}
		\label{lem:proj-fin-dim}
		Let $V$ be a normal vector field on $\Sigma$ and $\mathcal{K}$ a finite dimensional subspace of $L^2$. Then there exists $C$ such that 
		\begin{equation}
			\|\pi_\mathcal{K}(V)\|_{L^2} \leq C \|V\|_{L^1}.
		\end{equation}
	\end{lemma}
	\begin{proof}
		Let $U_i$ be an orthonormal basis of $\mathcal{K}$, then there exists $C$ such that $\|U_i\|_{C^{2,\alpha}} \leq C$ for all $i$. The trivial bound for integration then gives $|\langle V, U_i\rangle_{L^2}| \leq C \|V\|_{L^1}$. 
	\end{proof}
	
	\begin{proof}[Proof of Theorem \ref{thm:quant-rigidity-L2}]
		As before, we have $\pi_{\mathcal{K}_1}(Lh')=0$. By Proposition \ref{prop:thirdvar}, Lemma \ref{lem:proj-fin-dim} and finally (\ref{eq:ord3-L1}) we have that 
		\begin{equation}
		\begin{alignedat}{3}
			&\|U\|_{L^2}^3&&\leq{}&& C \|\pi_{\mathcal{K}_1}(-Lh'+\frac{1}{6}( \mathcal{D}^3\mathcal{\varphi}(U,U,U) + 3 \mathcal{D}^2\mathcal{\varphi}(U,W)))\|_{L^2} \\& \nonumber &&\leq{}&& C' \|-Lh'+\frac{1}{6}( \mathcal{D}^3\mathcal{\varphi}(U,U,U) + 3 \mathcal{D}^2\mathcal{\varphi}(U,W))\|_{L^1}
			\\&&&\leq {}&&C\|\phi_V\|_{L^1} + C( \||U|_2 |h'|_2 \|_{L^1}  + \|h'\|_{W^{2,2}}^2 + \|W\|_{W^{2,2}}^2 + \||h'|_2|W|_2 \|_{L^1}) \\&&&&&+C( \||U|_2^2 |h|_2\|_{L^1}+\| |U|_2 |h|_2^2 \|_{L^1} + \|h\|_{W^{2,3}}^3 + \|V\|_{W^{2,4}}^4).
		\end{alignedat}
		\end{equation}
		
		We now estimate the error terms on the right hand side, using the estimates from first and second order expansion. Recall that $\|U\|_{C^{2,\alpha}} \leq C \|V\|_{L^2}$ and $\|W\|_{C^{2,\alpha}} \leq C\|V\|_{L^2}^2$. 
		
		Then we have:
		
		\begin{alignat*}{3}
		& \||U|_2 |h'|_2 \|_{L^1} &&\leq \|U\|_{W^{2,2}}\|h'\|_{W^{2,2}} &&\leq C\|V\|_{L^2}(\|V\|_{L^2}^3 + \|\phi_V\|_{L^2}),\\
		&\|h'\|_{W^{2,2}}^2 &&\leq C (\|V\|_{L^2}^6 + \|\phi_V\|_{L^2}^2),\\
		&\|W\|_{W^{2,2}}^2 &&\leq C \|V\|_{L^2}^4,\\
		& \||h'|_2 |W|_2\|_{L^1} &&\leq \|W\|_{W^{2,2}} \|h'\|_{W^{2,2}} &&\leq C\|V\|_{L^2}^2 (\|V\|_{L^2}^3 + \|\phi_V\|_{L^2}),\\
		& \||U|_2^2 |h|_2\|_{L^1} &&\leq \|U\|_{W^{2,4}}^2\|h\|_{W^{2,2}}&& \leq C \|V\|_{L^2}^2 (\|V\|_{L^2}^2 + \|\phi_V\|_{L^2}) ,\\
		& \||U|_2 |h|_2^2\|_{L^1}&& \leq  \|U\|_{W^{2,2}} \|h\|_{W^{2,4}}^2 &&\leq C\|V\|_{L^2} (\|V\|_{L^2}^4 + \|\phi_V\|_{L^2}^2) ,\\
		&\|h\|_{W^{2,3}}^3 &&\leq C \|h\|_{W^{2,4}}^3 && \leq C'(\|V\|_{L^2}^6 + \|\phi_V\|_{L^2}^3),\\
		& \|V\|_{W^{2,4}}^4 && \leq C(\|V\|_{L^2}^4 + \|\phi_V\|_{L^2}^4). \end{alignat*}
		
		All the mixed terms are dominated by $\|\phi_V\|_{L^2}$, so taking the dominant power of $V$ and estimating $\|\phi_V\|_{L^1} \leq C\|\phi_V\|_{L^2}$, we deduce that
		\begin{equation}
			\|U\|_{L^2}^3 \leq C( \|V\|_{L^2}^4 + \|\phi_V\|_{L^2}). 
		\end{equation}
		Finally, we have 
		\begin{equation}
			\begin{split}
				\|V\|_{L^2}^3 &\leq C( \|U\|_{L^2}^3+ \|W\|_{L^2}^3 + \|h'\|_{L^2}^3) 
				\\&\leq C( \|V\|_{L^2}^4 + \|\phi_V\|_{L^2} + \|V\|_{L^2}^6 + \|V\|_{L^2}^9 + \|\phi_V\|_{L^2}^3). 
			\end{split}
		\end{equation}
		Taking only the dominant power of $\phi_V$ and absorbing higher powers of $V$ into the left hand side, we conclude the desired inequality
		\begin{equation}
			\|V\|_{L^2}^3 \leq C\|\phi_V\|_{L^2}. 
		\end{equation}
	\end{proof}
	
	We now complete the proof of the distance \L ojasiewicz inequality \ref{thm:quant-rigidity-intro}. 
	
	\begin{proof}[Proof of Theorem \ref{thm:quant-rigidity-intro}, Sobolev estimate]
		As before, let $\epsilon_0$ be small enough that both Theorems \ref{thm:quant-rigidity} and \ref{thm:quant-rigidity-L2} hold. Take $\epsilon'_0$ as in Lemma \ref{lem:change-base} and $\delta=\delta(\epsilon'_0)$ as in Lemma \ref{lem:slice}; then there is a rotation $\mathcal{R}$ such that $\mathcal{R}\cdot \Sigma'$ is the graph of some $V\in \mathcal{K}_0^\perp$ with $\|V\|_{C^1} <\epsilon_1$, and hence $\|V\|_{C^{2,\alpha}} <\epsilon_0$. We conclude from Theorem \ref{thm:quant-rigidity-L2} that \[\|V\|_{L^2} \leq C\|\phi_V\|_{L^2}^\frac{1}{3}.\] \end{proof}

	\subsection{Gradient \L ojasiewicz}
	
	We now prove the gradient \L ojasiewicz Theorem \ref{thm:gradient-L} by expansion of $\mathcal{F}$, similar to \cite[Proposition 6.5]{CM19}.

	\begin{proof}[Proof of Theorem \ref{thm:gradient-L}]
		By Theorem \ref{thm:quant-rigidity-intro}, up to a rotation we may assume $\Sigma'=\Sigma_V$ is the graph of a normal vector field $V\in \mathcal{K}_0^\perp$ with $\|V\|_{C^{2,\alpha}} <\epsilon_0$ and $\|V\|_{C^{2,\alpha}} \leq C \|\phi\|_{C^{0,\alpha}}^\frac{1}{3}$. Here $\phi=\phi_{\Sigma'}=\phi_V$. Consider now the quantity
		\begin{equation}
			\begin{split}
			\label{eq:gradL1}
				\mathcal{E}_3:&=  \int_0^1 ds \left(\langle V, \phi_{sV}\rangle_{L^2(\Sigma_{sV})}  +\langle V, sLV - \frac{s^2}{2} \mathcal{D}^2\mathcal{\varphi}(V,V)\rangle_{L^2}\right)
				\\&=\mathcal{F}(\Sigma_V) - \mathcal{F}(\Sigma) + \frac{1}{2}\langle V,LV\rangle_{L^2} -\frac{1}{6} \langle V, \mathcal{D}^2\mathcal{\varphi}(V,V)\rangle_{L^2}. \end{split}
		\end{equation}
		
		As before, decompose $V= U+h$ where $U\in \mathcal{K}$ and $h\in \mathcal{K}^\perp$. By Proposition \ref{prop:ord2-est-L2} and Theorem \ref{thm:quant-rigidity-L2}, we have $\|U\|_{C^{2,\alpha}} \leq C \|V\|_{W^{2,6}} \leq C' \|\phi_V\|_{L^2}^\frac{1}{3}$ and $\|h\|_{W^{2,4}} \leq C \|\phi_V\|_{L^2}^\frac{2}{3}$. Also by the proof of Proposition \ref{prop:ord1-est-L2} we have $\|Lh\|_{L^2} \leq C\|\phi_V\|_{L^2}^\frac{2}{3}.$
		
		Since $LU=0$, by self-adjointness of $L$ we have $\langle V,LV\rangle_{L^2} = \langle h, Lh\rangle_{L^2}$. It follows that \begin{equation}\label{eq:gradL2} |\langle V,LV\rangle_{L^2}| \leq C \|\phi_V\|_{L^2}^\frac{4}{3}.\end{equation}
		
		Now by Proposition \ref{prop:d2PhiK1}, we have that $\langle U, \mathcal{D}^2\mathcal{\varphi}(U,U)\rangle=0$. Then \[\langle V, \mathcal{D}^2\mathcal{\varphi}(V,V)\rangle_{L^2} = \langle h, \mathcal{D}^2\mathcal{\varphi}(U,U)\rangle_{L^2} + 2\langle V, \mathcal{D}^2\mathcal{\varphi}(U,h)\rangle_{L^2}  + \langle V, \mathcal{D}^2\mathcal{\varphi}(h,h)\rangle_{L^2}.\] 
		It follows that
		\begin{equation}
		\label{eq:gradL3}
			\begin{split}
				|\langle V, \mathcal{D}^2\mathcal{\varphi}(V,V)\rangle_{L^2}|& \leq C( \|h\|_{L^2} \|U\|_{W^{2,4}}^2 + \|V\|_{L^2} \| |U|_2 |h|_2 \|_{L^2} + \|V\|_{L^2} \| h \|_{W^{2,4}}^2)
				\\&\leq  C' \|\phi_V\|_{L^2}^\frac{4}{3}. 
			\end{split}
		\end{equation}
		
		The measures on $\Sigma$ and $\Sigma_{sV}$ are uniformly equivalent up to $C\|V\|_{C^1}$. Therefore, defining \begin{equation}\mathcal{E}'_3 := \int_0^1 ds \langle V, \phi_{sV} +sLV - \frac{s^2}{2} \mathcal{D}^2\mathcal{\varphi}(V,V)\rangle_{L^2(\Sigma_{sV})},\end{equation}
		we have 
		\begin{equation}
		\label{eq:gradL4}
			|\mathcal{E}_3 - \mathcal{E}'_3| \leq C \||V| |V|_1 (|LV| + |V|_2^2) \|_{L^1} \leq C(\|V\|_{W^{1,4}}^2 \|LV\|_{L^2} + \|V\|_{W^{2,4}}^4) \leq C \|\phi\|_{L^2}^\frac{4}{3}. 
		\end{equation}
		
		Finally, arguing as in Lemma \ref{lem:ord2} for $s\in [0,1]$ we have
		\begin{equation} |\phi_{sV} +sLV - \frac{s^2}{2} \mathcal{D}^2\mathcal{\varphi}(V,V)| \leq \frac{Cs^3}{6} \|V\|_{W^{2,6}}^3,\end{equation}
		which implies that 
		\begin{equation}
		\label{eq:gradL5}
			|\mathcal{E}'_3| \leq C'  \|V\|_{C^{2,\alpha}}^4 \leq C'' \|\phi_V\|_{L^2}^\frac{4}{3}. 
		\end{equation}
		
		Combining (\ref{eq:gradL2}), (\ref{eq:gradL3}), (\ref{eq:gradL4}) and (\ref{eq:gradL5}) using (\ref{eq:gradL1}) finally gives the desired inequality
		\begin{equation}
			|\mathcal{F}(\Sigma_V)-\mathcal{F}(\Sigma)| \leq C \|\phi_V\|_{L^2}^\frac{4}{3}.
		\end{equation}
	\end{proof}

	\appendix
	
	\section{Third variation polynomial analysis}
	\label{sec:poly}
	
	In this appendix we describe the rational functions $Q_0, Q_2, Q_4$ appearing in the proof of Proposition \ref{prop:thirdvar}. 
	
	We have \[Q_4 = \frac{144 r_1^2 r_2^2}{(r_1^2+2)(r_1^2+6)(r_2^2+2)(r_2^2+6)(2r_1^2+r_1^2r_2^2+2r_2^2)} P_4,\] where 
	
	\begin{equation}
		\begin{split}
			P_4 = {} &r_1^9 r_2^3 + r_2^3 r_1^9 + \frac{4}{3} (r_1^8 r_2^4 + r_1^4 r_2^8) -\frac{1}{3}(r_1^7 r_2^5 + r_1^5 r_2^7) + \frac{4}{3} r_1^6 r_2^6 \\& + 6(r_1^9 r_2 + r_1 r_2^9) + 10 (r_1^8 r_2^2 + r_1^2 r_2^8) + \frac{2}{3}( r_1^7 r_2^3 + r_1^3 r_2^7) + \frac{28}{3} (r_1^6r_2^4 + r_1^4 r_2^6) + 12 r_1^5r_2^5 \\&+4(r_1^8 + r_2^8) + \frac{32}{3}(r_1^7 r_2 + r_1 r_2^7) - 4(r_1^6 r_2^2 + r_1^2 r_2^6) + \frac{112}{3}(r_1^5 r_2^3 + r_1^3 r_2^5) - 32 r_1^4 r_2^4 \\& + 16(r_1^5 r_2+r_1r_2^5) - 32 (r_1^4 r_2^2 + r_1^2 r_2^4) + 32 r_1^3 r_2^3.
		\end{split}
	\end{equation} 
	
	\begin{proof}[Proof of Claim 1] To show that $Q_4 \geq 0$ for $r_1,r_2\geq 0$ we need only absorb the negative terms in $P_4$. For this we use the elementary inequalities: $\frac{1}{3} r_1^7 r_2^5 \leq \frac{1}{6} r_1^6 r_2^6 +  \frac{1}{6}r_1^8 r_2^4$; $4r_1^6 r_2^2 \leq 2 r_1^4 r_2^4 + 2r_1^8$; $32 r_1^4 r_2^2 \leq 16 r_1^3 r_2^3 + 16 r_1^5 r_2$; and the same inequalities with $r_1,r_2$ swapped. 
	\end{proof}
	
	We have \[Q_0 = \frac{48 r_1^2 r_2^2}{(r_1^2+2)^2(r_2^2+2)(r_2^2+6)(2r_1^2 + r_1^2 r_2^2 + 2r_2^2)} P_0,\] where 
	\begin{equation}
		\begin{split}
			P_0 = {} &3r_1^8 r_2^2 + 11 r_1^2 r_2^8 - r_1^7 r_2^3 -9r_1^3 r_2^7 + 2r_1^6 r_2^4+2r_1^4 r_2^6 + 18r_1^5 r_2^5 \\&+ 6r_1 r_2^9 + 6r_1^8 -14 r_1^7 r_2 -6 r_1r_2^7 -6 r_1^6 r_2^2 -22 r_1^2 r_2^6 + 34 r_1^5 r_2^3 +42 r_1^3 r_2^5 - 40 r_1^4 r_2^4. 
		\end{split}
	\end{equation} 
	
	\begin{proof}[Proof of Claim 2]
		We show that $P_0\geq 1024$ for $r_1,r_2\geq \sqrt{2}$, using the change of variable $r_i = s_i + \sqrt{2}$. With this substitution the polynomial becomes
		\begin{equation}
			\begin{split}
				P_0 = & {} 1024 + 3s_1^8 s_2^2 - s_1^7 s_2^3 + 2s_1^6 s_2^4 + 18 s_1^5 s_2^5 + 2s_1^4 s_2^6 -9s_1^3 s_1^7 + 11 s_1^2 s_2^8 + 6s_1 s_2^9 \\&+ 6\sqrt{2} s_1^8 s_2 + 21\sqrt{2} s_1^7 s_2^2 + \sqrt{2} s_1^6 s_2^3 + 102\sqrt{2} s_2^5 s_2^4 + 102\sqrt{2} s_1^4 s_2^5 \\&- 55 \sqrt{2} s_2^3 s_2^6 + 61\sqrt{2} s_1^2 s_2^7 + 76\sqrt{2} s_1 s_2^8 + 6\sqrt{2} s_2^9 \\& + 12 s_1^8  + 76 s_1^7 s_2 + 144 s_1^6 s_2^2 + 448 s_1^5 s_2^3 + 908 s_1^4 s_2^4 + 120 s_1^3 s_2^5 + 240 s_1^2 s_2^6 + 724 s_1 s_2^7 + 136 s_2^8 \\&+ 80\sqrt{2} s_1^7 + 200 \sqrt{2} s_1^6 s_2 + 780\sqrt{2} s_1^5 s_2^2 + 2060\sqrt{2} s_1^4 s_2^3 + 1540 \sqrt{2} s_1^3 s_2^4 + 596\sqrt{2}s_1^2 s_2^5 \\& + 1792\sqrt{2} s_1s_2^6 + 632\sqrt{2} s_2^7 + 444 s_1^6 +1116s_1^5s_2 + 5220 s_1^4 s_2^2 + 7320 s_1^3s_2^3 \\& + 3860 s_1^2 s_2^4 + 5708s_1s_2^5 + 3212 s_2^6 + 788\sqrt{2} s_1^5 + 2948\sqrt{2} s_1^4 s_2 + 8856 \sqrt{2} s_1^3 s_2^2 \\& + 7640\sqrt{2} s_1^2s_2^3 + 6932\sqrt{2} s_1 s_2^4 + 5092\sqrt{2} s_2^5 + 2256 s_1^4 + 9648 s_1^3 s_2 + 17056 s_1^2 s_2^2 \\& + 13936 s_1 s_2^3 + 10864 s_2^4 + 2608\sqrt{2}s_1^3 + 9104 \sqrt{2} s_1^2 s_2 + 10832 \sqrt{2} s_1 s_2^2 + 8176\sqrt{2} s_2^3 \\& + 4208 s_1^2 + 10016 s_1 s_2 + 8816 s_2^2 + 2048\sqrt{2} s_1 + 3072\sqrt{2} s_2. 
			\end{split}
		\end{equation} 
		
		Now we can again absorb the negative terms, using the elementary inequalities: 
		\[ s_1^7 s_2^3 \leq \frac{1}{2} s_1^8 s_2^2 + \frac{1}{2} s_1^6 s_2^4 ;\]
		\[ 9s_1^3 s_2^7  \leq 2s_1^4 s_2^6 + \frac{9}{4} s_1^2 s_2^8 ;\]
		\[ 55\sqrt{2} s_1^3 s_2^6 \leq  \frac{55\sqrt{2}}{2}( s_1^4 s_2^5 + s_1^2 s_2^7).\]
		
	\end{proof}
	
	For completeness, we record that \[Q_2 = \frac{144 r_1^2 r_2^2}{(r_1^2+2)(r_1^2+6)(r_2^2+2)(r_2^2+6)(2r_1^2+r_1^2r_2^2+2r_2^2)} P_2,\] where
	\begin{equation}
		\begin{split}
			P_2 = {} & r_1^8 r_2^2 +  r_1^2 r_2^8 -\frac{7}{3} (r_1^7 r_2^3 + r_1^3 r_2^7) -2(r_1^6 r_2^4 + r_1^4 r_2^6) +\frac{20}{3} r_1^5 r_2^5 \\& +2(r_1^8 + r_2^8) - \frac{50}{3} (r_1^7 r_2 + r_1r_2^7) - 18(r_1^6 r_2^2 + r_1^2 r_2^6) + \frac{38}{3} (r_1^5 r_2^3 + r_1^3 r_2^5) -24 r_1^4 r_2^4 \\& -16(r_1^5 r_2 + r_1 r_2^5) + 32 (r_1^4 r_2^2 + r_1^2 r_2^4) -32 r_1^3 r_2^3. 
		\end{split}
	\end{equation} 
	
	\bigskip

	\bibliographystyle{amsalpha}
	\bibliography{rigidity}

\end{document}